\documentclass{SCAE}%SCAEOL for online version; SCAE for publication version; SCAES for the paper dedicated to somebody.
\numberwithin{equation}{section}
\usepackage{amsmath,amsthm,amssymb}
\usepackage{subfig}
\usepackage{graphicx,color}
\usepackage{amsrefs}

\newcommand{\mc}[1]{\mathcal{#1}}

\newcommand{\ud}{\,\mathrm{d}}

\renewcommand{\abs}[1]{\left\vert#1\right\vert}
\newcommand{\norm}[1]{\left\Vert#1\right\Vert}

\newcommand{\wt}[1]{\widetilde{#1}}
\newcommand{\wh}[1]{\widehat{#1}}

\newcommand{\Or}{\mathcal{O}}

\newcommand*{\I}{\imath}

\newcommand{\RR}{\mathbb{R}}
\newcommand{\ZZ}{\mathbb{Z}}

\newcommand{\REV}[1]{{#1}}

\newtheorem*{theorem*}{Theorem}
\newtheorem{prop}{Proposition}[section]

\begin{document}
\Year{2016} %
\Month{January}
\Vol{59} %
\No{1} %
\BeginPage{1} %
\EndPage{XX} %
\AuthorMark{LIN L. and LU J.}
\ReceivedDay{November 17, 2014}
\AcceptedDay{January 22, 2015}
\PublishedOnlineDay{; published online January 22, 2016}
\DOI{10.1007/s11425-000-0000-0} % The author doesn't need fill in it.
%\Year{2016} %
%\Month{}
%\Vol{} %
%\No{} %
%\BeginPage{} %
%\EndPage{} %
%\AuthorMark{LIN L. and LU J.}
%\ReceivedDay{}
%\AcceptedDay{}
%\PublishedOnlineDay{}
%\DOI{} % The author doesn't need fill in it.

\title[Decay estimates of discretized Green's functions]{
  Decay estimates of discretized Green's functions for Schr\"odinger
  type operators}{}

\author[1]{LIN Lin}{}
\author[2]{LU Jianfeng}{}

\address[{\rm1}]{Department of Mathematics, University of California,
Berkeley \\
and 
Computational Research Division, Lawrence Berkeley National
Laboratory, Berkeley CA 94720 USA}
\address[{\rm2}]{Department of Mathematics, Department of Physics, and
Department of Chemistry, \\
Duke University, Box 90320, Durham NC 27708 USA}
\Emails{linlin@math.berkeley.edu,jianfeng@math.duke.edu}

\maketitle

 {\begin{center}
\parbox{14.5cm}{\begin{abstract}
 For a sparse non-singular matrix $A$, generally $A^{-1}$ is a dense
  matrix. However, for a class of matrices, $A^{-1}$ can be a matrix
  with off-diagonal decay properties, \textit{i.e.,} $\lvert
  A^{-1}_{ij}\rvert$ decays fast to $0$ with respect to the increase
  of a properly defined distance between $i$ and $j$.  Here we
  consider the off-diagonal decay properties of discretized Green's
  functions for Schr\"odinger type operators. We provide decay
  estimates for discretized Green's functions obtained from the finite
  difference discretization, and from a variant of the pseudo-spectral
  discretization. \REV{The asymptotic decay rate in our estimate} is
  independent of the domain size and of the discretization
  parameter. We verify the decay estimate with numerical results for
  one-dimensional Schr\"odinger type operators.\vspace{-3mm}
\end{abstract}}\end{center}}

%  Keyword is required.

%  \subjclass is required.

\renewcommand{\baselinestretch}{1.2}
\begin{center} \renewcommand{\arraystretch}{1.5}
{\begin{tabular}{lp{0.8\textwidth}} \hline \scriptsize
{\bf Citation:}\!\!\!\!&\scriptsize LIN L. and LU J. \makeatletter\@titlehead.
Sci China Math, 2016, 59,
 %\@Year, \@Vol: \@BeginPage--\@EndPage,
 doi:~\@DOI\makeatother\vspace{1mm}
\\
\hline
\end{tabular}}\end{center}

%%%%%%%%%%%%%%%%%%%%%%%%%%%%%%%%%%%%%%%%%%%%%%%%%%%%%%%%%%%%
%% Text of article.
%%%%%%%%%%%%%%%%%%%%%%%%%%%%%%%%%%%%%%%%%%%%%%%%%%%%%%%%%%%%
%    Section headings
\baselineskip 11pt\parindent=10.8pt  \wuhao

\section{Introduction}

Consider the following Schr\"odinger type partial differential equation
\begin{equation}
  \lambda G(x,y) - (-\Delta + V(x)) G(x,y) = \delta(x-y), \quad x,y\in
  \Omega\subset \mathbb{R}^{d}.
  \label{eqn:problem}
\end{equation}
Here $\Omega=[0,L)^{d}$ is a cubic domain with periodic boundary
conditions. $V(x)$ is a real, smooth potential function. Here $\delta(x)$
is the Dirac $\delta$-distribution, \REV{and} $\lambda\in\mathbb{C}$ is in the
resolvent set of the Hamiltonian operator $H:=-\Delta + V(x)$. Then
$G$ is called the Green's function of $\lambda - H$. It can be shown
that $G$ decays exponentially to zero along the off-diagonal
direction.  Roughly speaking, if the domain size $L$ is large enough,
for each fixed $y\in\Omega$, the magnitude of $G(x,y)$ decays exponentially as
$d(x,y)$ increases, where $d(x,y)$ is the distance between
$x,y\in \Omega$ interpreted in the periodic sense. Furthermore, such
decay rate is independent of the domain size $L$. In fact, the following
theorem has been established in the previous work~\cite{ELu:11} by one
of the authors for Hamiltonian operators defined on $\RR^d$ (and hence
contains the current periodic case as a special situation).

\begin{theorem*}[E-Lu 2011~\cite{ELu:11}] Assume
  $\lambda$ lies in the resolvent set of $H$, then there exist constants
  $\gamma,C>0$ such that
  \[
  \sup_{y\in \mathbb{R}^{d}}\norm{e^{\gamma ((\cdot -y)^2+1)^{1/2}}
  (\lambda-H)^{-1} e^{-\gamma ((\cdot -y)^2+1)^{1/2}}}_{\mathcal{L}(L^2)}\le C, 
  \]
  where the exponential functions are viewed as multiplication
  operators.
\end{theorem*}

The decay property is a powerful tool for
designing efficient numerical methods, such as the sparse approximate
inverse preconditioner (AINV)~\cite{BenziMeyerTuma1996,BenziTuma1998},
incomplete $LU$ and Cholesky type factorization~\cite{Saad1994}, and
localized spectrum slicing~\cite{Lin2014}.  It also has profound
implication in science and engineering applications.  In quantum
physics literature, the exponential decay property of Green's functions
and related physical quantities is referred to as
the ``near-sightedness principle''~\cite{Kohn1996,ProdanKohn2005} of
electronic matters. A variety of ``linear scaling'' methods for
solving Kohn-Sham density functional theory~\cite{KohnSham1965} for gapped systems have
been proposed in the past two decades, such as the divide-and-conquer
method~\cite{Yang1991,ChenLu2014}, and remains an active research
field (see e.g., the review articles
\cite{Goedecker1999,BowlerMiyazaki2012,BenziBoitoRazouk2013}).

In order to develop efficient numerical methods taking advantage of the
decay property of Green's functions, we require the Schr\"odinger
operator to be discretized using a certain numerical scheme. It turns
out that not all numerical discretization schemes lead to
\textit{discretized}
Green's function with exponentially decaying  off-diagonal elements.
This paper is concerned with demonstrating that discretized
Green's functions obtained from proper numerical schemes also have decay
properties, either exponentially or super-algebraically.  
We obtain \REV{decay estimates of which} the decay rate
is asymptotically independent of the discretization parameter (e.g., the
grid size in finite difference discretization), and of the domain
size.  To the best of our knowledge, such results were not known in
previous literature.

\subsection*{Previous work:}

The exponential decay properties of Green's functions in the
continuous setup and the related exponential decay of
eigenfunctions of elliptic operators have been widely studied
(see e.g., \cites{Agmon:65, CombesThomas:73, Simon:83, ELu:11, ELu:13}). 

In the discretized setup, the exponential decay of discretized Green's
functions was first studied in~\cite{Demko1977,DemkoMossSmith1984} for
the matrix inverse $A^{-1}$, where $A$ is assumed to be a banded,
positive definite matrix. In order to generalize from banded matrices
to general sparse matrices, decay properties should be defined using
geodesic distances of the graph induced by $A$. These techniques have
been used in~\cite{BenziBoitoRazouk2013,BenziRazouk2007} and
references therein, for demonstrating the decay properties of e.g.,
Fermi-Dirac operators in electronic structure theory. This type of
decay estimate relies on the following facts: 1) a complex analytic
function such as $z^{-1}$ where $z$ belongs to a simply connected
complex domain away from $0$, can be efficiently expanded using
polynomials of controllably low degrees, and 2) when the matrix size
is sufficiently large, a finite term polynomial of a sparse matrix
remains a sparse matrix. This argument can be further generalized to
non-sparse matrices with exponentially decaying off-diagonal elements,
and is not restricted to Schr\"odinger type operators in
Eq.~\eqref{eqn:problem}.  It can be shown that the exponent for the
exponential decay estimate is \REV{bounded by a constant while
  increasing} the domain size $L$. However, the decay rate is \REV{not
  uniform} with respect to the refinement of the discretization
parameter.

Simply speaking, the reason why the general argument above cannot
produce \REV{optimal} decay estimates with increasingly refined discretization
is as follows.  Due to the presence of the Laplacian operator
$-\Delta$, $H$ is an unbounded operator. The spectral radius of the
discretized $H$ increases as the discretization refines. For instance,
for finite difference discretization with uniform grid spacing $\Delta
x$, the spectral radius of the discretized $H$ increases as
$\Or(\Delta x^{-2})$. As a result, the order of polynomials needed to
accurately approximate the complex analytic function such as $z^{-1}$
increases as $\Or(\Delta x^{-2})$, and the decay rate deteriorates.
In the limit when the $\Delta x\to 0$, it can be shown that the
exponential decay rate in the ``physical'' space approaches $0$.
However, as $\Delta x\to 0$ the discretized Green's function should
well approximate the continuous Green's function up to consistency
error, and hence should share the decay property of the continuous
Green's functions.  The discrepancy between the decay properties of
the discrete and continuous versions of the Green's functions is due to
the fact that such decay estimates for discretized Green's function
provides only \textit{a lower bound} of the exponential decay rate,
and such lower bound is not \REV{optimal}.  Therefore, this type of estimate
is mostly suitable for discretized $H$ with relatively small spectral
radius, i.e., discretization with low to medium accuracy. It is
desirable to have a better estimate which correctly captures the decay behavior
for all accuracy level.

\subsection*{Our contribution:}

In this paper we provide decay estimates of discretized Green's
functions for Sch\"odinger type operators. \REV{The decay rate of our
  estimates} is asymptotically independent of both the domain size and
the discretization parameter. We demonstrate the decay estimate for
two types of discretization: finite difference discretization and a
variant of the pseudo-spectral discretization.  \REV{Our result is
  explicitly stated for one-dimensional Schr\"odinger type
  operators. However, generalization to Schr\"odinger type operators
  in higher dimensions is straightforward with necessary notational
  changes.}

For the finite difference discretization, our argument is analogous to
the decay estimate of continuous Green's
functions~\cite{ELu:11}. Compared to the general argument in
e.g.,~\cite{Demko1977,BenziBoitoRazouk2013} based on matrix sparsity,
our method specifically exploits the structure of the discretized
Laplacian operator. More specifically, we use the discretized Green's
function, which is a matrix of bounded spectral radius, to control
other operators with diverging spectral radius. Such operators include
the discretized first and second order differential operators. We find
that the discretized Green's function decays exponentially along the
off-diagonal direction (see Theorem~\ref{thm:fdexp}).

For the pseudo-spectral discretization, the off-diagonal elements of the
discretized $H$ only decay
polynomially. We verify
numerically that the corresponding discretized Green's function
\textit{does not} decay exponentially along the off-diagonal
direction.  However, if we systematically mollify the high end of the spectrum of the
discretized Laplacian operator, the resulting discretized $H$ will
decay super-algebraically along the off-diagonal direction. We refer
to this scheme as the mollified pseudo-spectral method (mPS).  We
demonstrate that the off-diagonal elements of the discretized Green's function corresponding to the
mPS discretization decay super-algebraically. For any given
polynomial order, the decay rate does not depend on the domain length
or the discretization parameter (see Theorem~\ref{thm:Dmboundfinish}
and Theorem~\ref{thm:Dmboundgeneral}). The proof of this result relies
on the discrete version of the relation between the regularity of the
Fourier space and the decay in the real space.

\subsection*{Notation:}
The following notation is used throughout the paper. With some abuse
of notation, unless otherwise clarified, the symbol $H$ denotes both
the continuous operator $-\Delta+V(x)$, and its discretized matrix, for
both the finite difference discretization and for the pseudo-spectral
type discretization.
Similarly $G$ denotes both the continuous Green's function for the
operator $\lambda-H$ and its discretized matrix.  $\I$ stands for the
imaginary unit. The complex conjugate of a complex number $f$ is
denoted by $f^{*}$. The identity matrix is denoted by $I$. When the identity matrix is multiplied by a scalar
$\lambda$, the matrix $\lambda I$ is also denoted by $\lambda$ for
simplicity, unless otherwise clarified.

For simplicity of the notation, we will restrict ourselves to the
cases that the computational domain is an interval $\Omega=[0, L)$ in one
spatial dimension. The extension to higher spatial dimensional
rectangular computational domain is straightforward.  The
computational domain is discretized by $N$ equispaced grid points:
$\mc{X} = \{x_i \mid x_i = i \Delta x, i = 0, 1, \ldots, N-1\}$, where
$\Delta x = L/N$ is the grid size. 

Throughout this paper, since
we are only interested in the asymptotic decay behavior, we will
assume that $L \geq 1$ and also without loss of generality $\Delta x
\leq 1$.

For a lattice function $f: \mc{X} \to \RR$, we define its
$L^2(\mc{X})$ norm as
\begin{equation}
  \norm{f}_{L^2(\mc{X})}^2  = \Delta x \sum_{x \in \mc{X}} \abs{f(x)}^2,  
  \label{eqn:normL2}
\end{equation}
so that as $\Delta x \to 0$, it converges to the continuous $L^2$ norm
on $[0, L)$. Similarly the $L^{\infty}(\mc{X})$ norm is defined as
\begin{equation}
  \norm{f}_{L^\infty(\mc{X})}  = \max_{x\in\mc{X}} \abs{f(x)}.  
  \label{eqn:normLinf}
\end{equation}
For simplicity of the notation, we will use $\norm{f}_{2}$ and
$\norm{f}_{\infty}$ interchangeably with $\norm{f}_{L^{2}(\mc{X})}$ and
$\norm{f}_{L^{\infty}(\mc{X})}$, respectively.

We will focus on periodic boundary condition, so that a function
$f(x)$ defined on the finite lattice $\mc{X}$ can be extended to a
periodic function on the
infinite lattice $\Delta x \ZZ$ such that $f( x + L ) = f(x), x\in
\mc{X}$. 

We will use $C$ for generic absolute constants whose value may change
from line to line. Specific constants are denoted as e.g., $C_m$ where
the subscript $m$ indicates the dependence of the constant on the
parameter $m$.

\subsection*{Organization:}
This paper is organized as follows. We estimate the decay rate for the
finite difference discretization in section~\ref{sec:fd}, and the decay
rate for the mollified pseudo-differential discretization in
section~\ref{sec:ps}. Numerical results demonstrating the decay
rate is provided in section~\ref{sec:numer}, and we conclude in
section~\ref{sec:conclusion}.

\section{Finite difference discretization}\label{sec:fd}
 
In this paper we focus on the second order finite difference
discretization, and it is possible to generalize the analysis to higher
order finite difference discretization schemes.
We define the forward and backward difference operators for $x \in
\mc{X}$, respectively as
\begin{align}
  & \bigl( \mc{D}^+ f \bigr)(x) = \frac{1}{\Delta x} \bigl( f(x + \Delta x) - f(x) \bigr), \\
  & \bigl( \mc{D}^- f \bigr)(x) = \frac{1}{\Delta x} \bigl( f(x) - f(x -
    \Delta x) \bigr).
\end{align}

The Hamiltonian operator in the second order 
finite difference discretization is 
\begin{equation}
  H = \mc{D}^+ \mc{D}^- + V. 
\end{equation}
Here the potential function $V(x)$ is discretized into a lattice
function $V:\mc{X} \to \RR$ with bounded $L^{\infty}(\mc{X})$ norm. With
some abuse of notation, unless otherwise clarified, we use $\Delta=\mc{D}^+ \mc{D}^-$ 
to denote  the discretized Laplacian operator as well.

Since periodic boundary condition is used, the natural distance
between two grid points $x,y\in \mc{X}$ is the periodic distance 
\begin{equation*}
  \wt{d}_{L}(x, y) = \min
  \bigl\{ \abs{ x - y - L k }, \; k \in \ZZ \bigr\}.
\end{equation*}
As in the continuous case, we need to
mollify the distance to remove singularities as
\begin{equation}
  d_{L}(x, y) :=  d_{\max} - \Bigl( \bigl[ d_{\max} -   (\wt{d}_{L}(x, y)^2 + 1)^{1/2} \bigr]^2 + 1 \Bigr)^{1/2},
\end{equation}
where $d_{\max} = \max\, (\wt{d}_{L}(x, y)^2 + 1)^{1/2} = (L^2/4 +
1)^{1/2}$.  Note that the slightly complicated looking formula is due to the
necessity of mollification when $\wt{d}_{L}$ is either $0$ or
$L/2$.  \REV{Fig.~\ref{fig:dL} gives an example of the distance function
$d_{L}(x,0)$ with $L=40$}.

\begin{figure}[h]
  \begin{center}
    \includegraphics[width=0.3\textwidth]{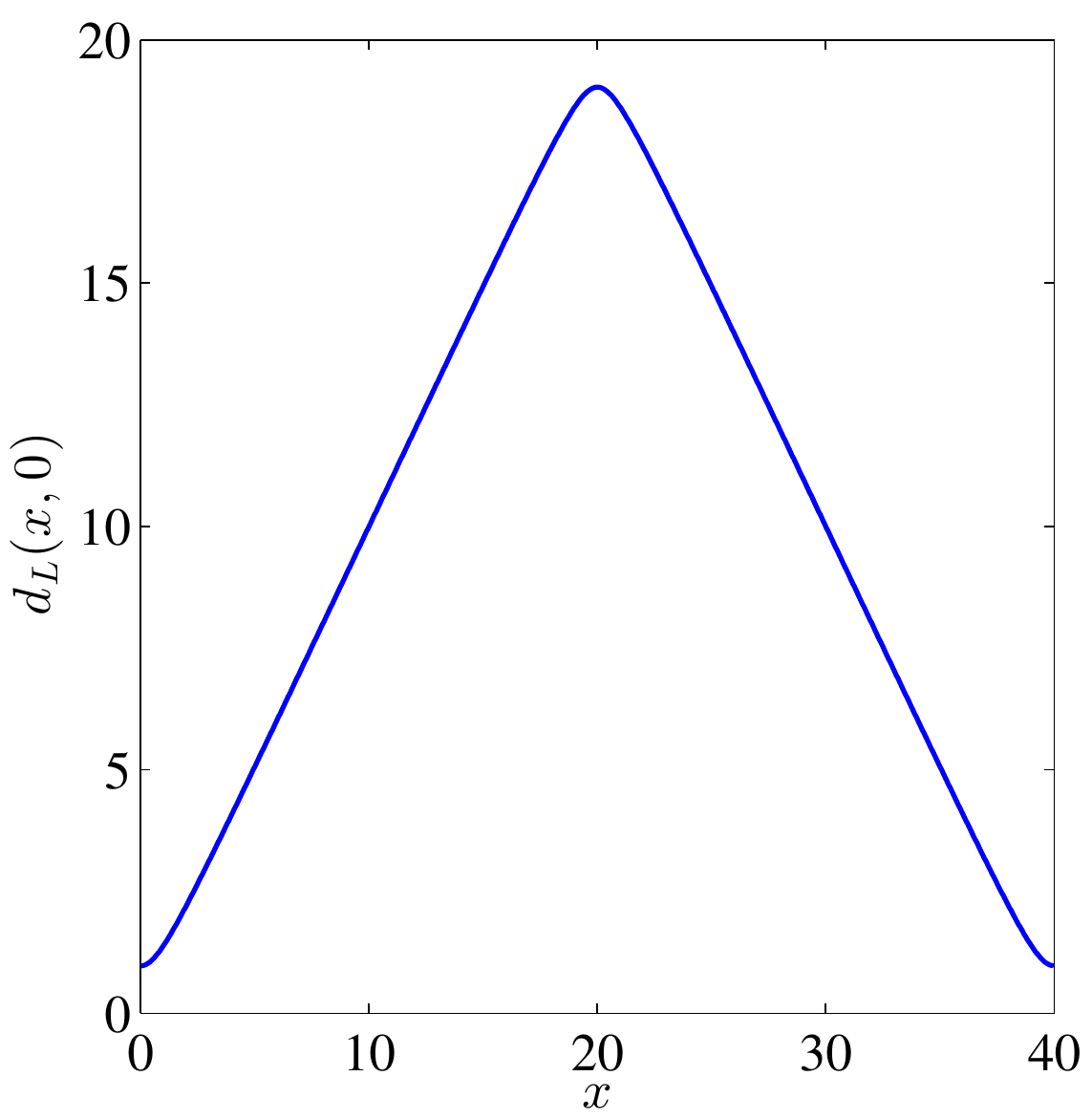}
  \end{center}
  \caption{An illustration of the smoothed distance function $d_{L}(x,0)$ with $L=40$.}
  \label{fig:dL}
\end{figure}

The following lemma collects the properties of $d_L$ that will be used
for proving Theorem~\ref{thm:fdexp}.
\begin{lemma}\label{lem:dl}
  For fixed $y \in \mc{X}$, the function $d_L(\cdot, y)$ is twice
  continuous differentiable and the derivatives are bounded uniformly
  in $L$ and $\Delta x$.
\end{lemma}

\begin{proof}
  We include the elementary proof here for completeness.  We fix $y = 0$
  without loss of generality, and have
  \begin{equation*}
    \begin{aligned}
      d_{L}(x, 0) & = d_{\max} - \Bigl( \bigl[ d_{\max} -   \bigl(
    \min(\abs{x}, \abs{x - L})^2 + 1 \bigr)^{1/2} \bigr]^2 + 1 \Bigr)^{1/2}  \\
    & = 
    \begin{cases}
      d_{\max} - \Bigl( \bigl[ d_{\max} - \bigl( x^2 + 1
      \bigr)^{1/2} \bigr]^2 + 1 \Bigr)^{1/2}, & x \in [0, L/2), \\
      d_{\max} - \Bigl( \bigl[ d_{\max} - \bigl( (L - x)^2 + 1
      \bigr)^{1/2} \bigr]^2 + 1 \Bigr)^{1/2}, & x \in [L/2, L).
    \end{cases}
    \end{aligned}
  \end{equation*}
  We calculate the derivative of $d_{L}(x, 0)$ in each interval as
  follows
  \begin{equation*}
    \frac{\partial d_{L}(x, 0)}{\partial x} = 
    \begin{cases}
      \dfrac{ x \Bigl(d_{\max} - \bigl( x^2 + 1
        \bigr)^{1/2}\Bigr)}{ \bigl(x^2 + 1\bigr)^{1/2} \Bigl( \bigl[
        d_{\max} - \bigl( x^2 + 1 \bigr)^{1/2} \bigr]^2 + 1
        \Bigr)^{1/2}}, &  x \in [0, L/2), \\
      \dfrac{ - (L - x) \Bigl(d_{\max} - \bigl( ( L - x)^2 + 1 \bigr)^{1/2}
        \Bigr)}{\bigl((L - x)^2 + 1\bigr)^{1/2} \Bigl( \bigl[
        d_{\max} - \bigl( (L - x)^2 + 1 \bigr)^{1/2} \bigr]^2 + 1
        \Bigr)^{1/2}}, & x \in [L/2, L).
    \end{cases}
  \end{equation*}
  In particular, it is continuous at $x = L/2$ and $x = 0$ (viewed
  as a periodic function on $[0, L)$). The expression also verifies
  that
  \begin{equation*}
    \biggl\lvert \frac{\partial d_{L}(x, 0)}{\partial x} \biggr\rvert \leq 1. 
  \end{equation*}

  To calculate the second order derivative, denote
  $\phi(t) = t / (t^2 + 1)^{1/2}$ and we write
  \begin{equation*}
    \frac{\partial d_{L}(x, 0)}{\partial x} = 
    \begin{cases}
      \phi(x) \phi\Bigl(d_{\max} - \bigl( x^2 + 1
      \bigr)^{1/2}\Bigr), &  x \in [0, L/2), \\
      - \phi(L - x) \phi\Bigl(d_{\max} - \bigl( ( L - x)^2 + 1
      \bigr)^{1/2} \Bigr), & x \in [L/2, L).
    \end{cases}
  \end{equation*}
  Hence, 
  \begin{equation*}
    \frac{\partial^2 d_{L}(x, 0)}{\partial x^2} = 
    \begin{cases}
      \begin{split}
        & \phi'(x) \phi\Bigl(d_{\max} - \bigl( x^2 + 1
        \bigr)^{1/2}\Bigr) \\
        & \qquad - \phi^2(x) \phi'\Bigl(d_{\max} - \bigl( x^2 + 1
        \bigr)^{1/2}\Bigr), 
      \end{split}
      &  x \in [0, L/2), \\
      \begin{split}
        & \phi'(L - x) \phi\Bigl(d_{\max} - \bigl( ( L - x)^2 + 1
        \bigr)^{1/2} \Bigr) \\
        & \qquad - \phi^2(L - x) \phi'\Bigl(d_{\max} - \bigl( ( L - x)^2
        + 1 \bigr)^{1/2} \Bigr),
      \end{split} & x \in [L/2, L).
    \end{cases}
  \end{equation*}
  Since $\phi'(t) = (t^2 + 1)^{-3/2}$, it is clear that the second order
  derivative is uniformly bounded. To check the continuity, it suffices to
  check $x = L/2$ and $x = 0$. We have 
  \begin{equation*}
    \begin{aligned}
      & \lim_{x \to 0 +} \frac{\partial^2 d_{L}(x, 0)}{\partial
        x^2} = \phi'(0) \phi(d_{\max} - 1) =
      \lim_{x \to L-}  \frac{\partial^2 d_{L}(x, 0)}{\partial x^2}, \\
      & \lim_{x \to L/2 -} \frac{\partial^2 d_{L}(x, 0)}{\partial
        x^2} = - \phi^2(L/2) \phi(0) = \lim_{x \to L/2 +} \frac{\partial^2
        d_{L}(x, 0)}{\partial x^2}, 
    \end{aligned}
  \end{equation*}
  where the second line uses that $d_{\max} = (L^2/4 + 1)^{1/2} =
  \lim_{x \to L/2} (x^2+1)^{1/2}$. 
\end{proof}

In order to prove Theorem~\ref{thm:fdexp}, we need 
the discrete version of the
Leibniz rule in the finite difference discretization. For any
$x\in\mc{X}$,
\begin{equation}
  \begin{split}
    \mc{D}^- (f g)(x) &= \frac{1}{\Delta x} \bigl( (fg)(x) - (fg)(x - \Delta x) \bigr) \\
    &= \frac{1}{\Delta x} \bigl(f(x) - f(x - \Delta x)\bigr) g(x)
    + \frac{1}{\Delta x} f(x - \Delta x) \bigl( g(x) - g(x - \Delta x)\bigr) \\
    &= \bigl( (\mc{D}^- f) g \bigr)(x) + f(x - \Delta x)
    \bigl(\mc{D}^- g\bigr) (x).
  \end{split}
  \label{eq:leibniz}
\end{equation}

\begin{theorem}\label{thm:fdexp}
  Assume that $(\lambda - H)^{-1}$ is bounded in the matrix 2-norm,
  then there exist constants $\gamma_0>0$ and $C$ such that for any
  $\Delta x \leq 1$, $L \geq 1$, and $\gamma \leq \gamma_0$,
  \begin{equation*}
    \sup_{y \in \mc{X}} \, \bigl\lVert 
    \exp(\gamma d_L(\cdot, y)) 
    (\lambda - H)^{-1} 
    \exp(-\gamma d_L(\cdot, y))  \bigr\rVert_{\mc{L}(L^2(\mc{X}))} \leq
    C, 
  \end{equation*}
  where $\exp(-\gamma d_L(\cdot, y))$ is understood as a
  multiplication operator: 
  \begin{equation*}
    \bigl( \exp(-\gamma d_L(\cdot, y)) f \bigr)(x) = 
    \exp(-\gamma d_L(x, y)) f(x), \quad x\in\mc{X}.
  \end{equation*}
  The definition for $\exp(\gamma d_L(\cdot, y))$ is similar.
\end{theorem}
\begin{proof}
  Notice first that 
  \begin{equation*}
    \exp(\gamma d_L(\cdot, y)) (\lambda - H)^{-1} \exp(-\gamma d_L(\cdot, y))
    = \bigl[ \exp(\gamma d_L(\cdot, y)) (\lambda - H) \exp(-\gamma
    d_L(\cdot, y)) \bigr]^{-1}.
  \end{equation*}
  Using the definition of $H$, we get
  \begin{equation*}
    \begin{aligned}
      \exp(\gamma d_L(\cdot, y)) (\lambda - H) \exp(-\gamma
      d_L(\cdot, y)) & = \exp(\gamma d_L(\cdot, y))
      (\lambda - V) \exp(-\gamma d_L(\cdot, y)) \\
      & \qquad + \exp(\gamma d_L(\cdot, y)) \Delta \exp(-\gamma
      d_L(\cdot, y)) \\
      & = (\lambda - V) + \exp(\gamma d_L(\cdot, y)) \Delta \exp(-\gamma
      d_L(\cdot, y))
    \end{aligned}
  \end{equation*}
  Explicit calculation using \eqref{eq:leibniz} for $\mc{D}^-$ and
  analogously for $\mc{D}^+$, we obtain
  \begin{equation*}
    \begin{aligned}
      \exp(\gamma d_L(\cdot, y)) \Delta & \exp(-\gamma d_L(\cdot,
      y)) = \exp(\gamma d_L(\cdot, y)) \mc{D}^+ \mc{D}^- \exp(-\gamma d_L(\cdot, y)) \\
      & = \Delta + \exp(\gamma d_L(\cdot, y)) \bigl[
      \mc{D}^- \exp(-\gamma d_L(\cdot, y)) \bigr] \mc{D}^- \\
      & \qquad +  \exp(\gamma d_L(\cdot, y)) \bigl[
      \mc{D}^+ \exp(-\gamma d_L(\cdot, y)) \bigr] \mc{D}^+ \\
      & \qquad + \exp(\gamma d_L(\cdot, y)) \bigl[ \Delta \exp(-\gamma
      d_L(\cdot, y)) \bigr].
    \end{aligned}
  \end{equation*}
  To control the lower order terms on the right hand side, we estimate
  \begin{multline*}
    \Bigl\lvert \bigl[ \mc{D}^+ \exp(-\gamma d_L(\cdot, y)) \bigr](x)
    \Bigr\rvert = \Bigl\lvert \frac{1}{\Delta x} \bigl[ e^{-\gamma
      d_L(x + \Delta x, y)}
    - e^{-\gamma d_L(x, y)} \bigr] \Bigr\rvert \\
    \leq \max_{t \in [0, \Delta x]} \Bigl\lvert
    \frac{\partial}{\partial t} e^{-\gamma d_L(x + t, y)} \Bigr\rvert
    \leq C \max \bigl( \gamma e^{-\gamma d_L(x, y)}, \gamma e^{-\gamma
      d_L(x+\Delta x, y)} \bigr) \leq C \gamma e^{\gamma \Delta x}
    e^{-\gamma d_L(x, y)},
  \end{multline*}
  where the first inequality follows from the mean value theorem, and the
  second inequality uses that $\abs{\partial_x d_L(x, y)}$ is
  uniformly bounded from Lemma~\ref{lem:dl}. The same bound also holds
  for $\mc{D}^- \exp(-\gamma d_L(\cdot, y))$.  For the second order
  difference
  \begin{equation*}
    \begin{aligned}
      \Bigl\lvert \bigl[ \mc{D}^+ \mc{D}^- \exp(-\gamma d_L(\cdot,
      y)) \bigr] (x) \Bigr\rvert 
      & = \biggl\lvert \frac{1}{\Delta x^2} \bigl[
      e^{-\gamma d_L(x + \Delta x, y)} - 2 e^{-\gamma d_L(x, y)} +
      e^{-\gamma d_L(x - \Delta x, y)} \bigr] \biggr\rvert \\
& = \frac{1}{\Delta x^2} \biggl \lvert
      \int_0^{\Delta x} \int_0^{\Delta x}
      \partial_t \partial_s \, e^{- \gamma d_L(x + (t - s), y)} \ud s \ud t \biggr\rvert \\
      & \leq \max_{(t, s) \in [0, \Delta x]^2} \bigl \lvert \partial_t \partial_s \, e^{- \gamma d_L(x + (t - s), y)} \bigr \rvert \\
      & \leq C (\gamma^2 + 2 \gamma ) e^{\gamma \Delta x} e^{-\gamma
        d_L(x, y)},
    \end{aligned}  
  \end{equation*}
  where we have used Lemma~\ref{lem:dl} in the last inequality.

  We thus have in summary for $\gamma$ sufficiently small (recall that
  $\Delta x \leq 1$)
  \begin{align*}
    & \norm{f^-}_{L^{\infty}(\mc{X})} := \bigl\lVert \exp(\gamma d_L(\cdot, y)) \bigl[ \mc{D}^-  \exp(-\gamma d_L(\cdot, y)) \bigr]
      \bigr\rVert_{L^{\infty}(\mc{X})}
      \leq C \gamma, \\
    & \norm{f^+}_{L^{\infty}(\mc{X})} := \bigl\lVert \exp(\gamma d_L(\cdot, y)) \bigl[ \mc{D}^+
      \exp(-\gamma d_L(\cdot, y)) \bigr]
      \bigr\rVert_{L^{\infty}(\mc{X})}
      \leq C \gamma, \\
    & \norm{g}_{L^{\infty}(\mc{X})} := \bigl\lVert \exp(\gamma d_L(\cdot, y)) \bigl[ \Delta
      \exp(-\gamma d_L(\cdot, y)) \bigr]
      \bigr\rVert_{L^{\infty}(\mc{X})} \leq C \gamma, 
   \end{align*}
   where we have introduced the short hand notation $f^{\pm}$ and
   $g$.   
  
   \smallskip 

   Recall the identity
   \begin{equation}\label{eq:neumann}
     \begin{aligned}
       \exp(\gamma d_L(\cdot, y)) & (\lambda - H) \exp(-\gamma
       d_L(\cdot, y)) \\
       & = (\lambda - V) + \Delta + 
       \bigl( f^{-} \mc{D}^- + f^{+} \mc{D}^+ \bigr) + g \\
       & = ( \lambda - H) + \bigl( f^{-} \mc{D}^- +
       f^{+} \mc{D}^+ \bigr) + g \\
       & = ( \lambda - H) \biggl[ I + \bigl( f^{-}
       \mc{D}^- + f^{+} \mc{D}^+ \bigr) (\lambda - H)^{-1} + g (\lambda - H)^{-1}
       \biggr]. 
     \end{aligned}
   \end{equation}
   Note that 
   \begin{equation}
     \begin{aligned}
       \norm{\mc{D}^- ( \lambda - H)^{-1}}_{\mc{L}(L^2(\mc{X}))} &
       \leq
       \norm{ \mc{D}^{-} (1 - \Delta)^{-1}}_{\mc{L}(L^2(\mc{X}))} \norm{(1 - \Delta) (\lambda - H)^{-1}}_{\mc{L}(L^2(\mc{X}))} \\
       & \leq C \norm{(1+\lambda-V-\lambda + H) (\lambda -
         H)^{-1}}_{\mc{L}(L^2(\mc{X}))}
       \\
       & \leq C \left(1 +
       (\abs{1+\lambda}+\norm{V}_{L^{\infty}(\mc{X})})\right)
       \norm{(\lambda - H)^{-1}}_{\mc{L}(L^2(\mc{X}))},
     \end{aligned}
     \label{eqn:bounddminus}
   \end{equation}
   and the same bound for $\mc{D}^+ (\lambda - H)^{-1}$. Here we
     have used the fact that $\norm{ \mc{D}^{-} (1 -
       \Delta)^{-1}}_{\mc{L}(L^2(\mc{X}))}$ is bounded uniformly with
     respect to $\Delta x \leq 1$, which can be directly verified by
     Fourier representation. 
     Thus by making
   $\gamma$ sufficiently small, the bounds on $f^{\pm}$ and $g$
   guarantee the invertibility of the last term on the right hand side
   of \eqref{eq:neumann} and the inverse is also bounded. The theorem
   is hence proved.
\end{proof}

As a corollary to Theorem~\ref{thm:fdexp}, we may infer the pointwise
decay property of the Green's function. Let us consider without loss
of generality a single column $g$ of the discretized Green's function,
which solves the equation
\begin{equation}
  ( \lambda - H) g = \frac{1}{\Delta x} e_1 
  \label{eqn:gfd}
\end{equation}
with $e_1 = (1, 0, 0, \ldots, 0)^{T}$.  Here the prefactor $1/\Delta
x$ on the right hand side of Eq.~\eqref{eqn:gfd} reflects the
normalization of the discrete Dirac $\delta-$distribution.  Thus $g =
\frac{1}{\Delta x} (\lambda - H)^{-1} e_1$. We estimate the
exponential decay rate of $g$ (in $L^2$ sense) according to
\begin{multline}
  \norm{e^{\gamma d_L(0, \cdot )} g}_2  =  \norm{e^{\gamma d_L(0, \cdot)} (\lambda - H)^{-1} \frac{e_1}{\Delta x} }_2
    =  \norm{e^{\gamma d_L(0, \cdot)} (\lambda - H)^{-1} (1 - \Delta) (1 - \Delta)^{-1} \frac{e_1}{\Delta x}}_2 \\
    \leq \norm{e^{\gamma d_L(0, \cdot)} (\lambda - H)^{-1} (1 -
    \Delta) e^{-\gamma d_L(0, \cdot)}}_{\mc{L}(L^2(\mc{X}))}
     \norm{e^{\gamma
        d_L(0, \cdot)} (1 - \Delta)^{-1} \frac{e_1}{\Delta x} }_2. 
  \label{eqn:boundgfd}
\end{multline}
The right hand side of~\eqref{eqn:boundgfd} is bounded because of the following two facts.
First, similar to Eq.~\eqref{eqn:bounddminus},
\begin{multline}
  \norm{e^{\gamma d_L(0, \cdot)} (\lambda - H)^{-1} (1 - \Delta) e^{-\gamma d_L(0, \cdot)}}_{\mc{L}(L^2(\mc{X}))} \\
  = \norm{e^{\gamma d_L(0, \cdot)} (\lambda - H)^{-1} (-(\lambda
    -H) + (\lambda + 1) - V) e^{-\gamma d_L(0, \cdot)}}_{\mc{L}(L^2(\mc{X}))} \\
  \leq 1 + \bigl( \abs{1+\lambda} + \norm{V}_{L^{\infty}(\mc{X})}
  \bigr) \norm{e^{\gamma d_L(0, \cdot)} (\lambda - H)^{-1} e^{-\gamma
      d_L(0, \cdot)}}_{\mc{L}(L^2(\mc{X}))}.
\end{multline}
Second, $ \norm{e^{\gamma d_L(0, \cdot)} (1 - \Delta)^{-1} \frac{e_1}{\Delta
    x} }_2$
is bounded for sufficiently small $\gamma$, which can be verified by a
direct calculation using the explicit discrete Green's function
for $(1 - \Delta)^{-1}$ of Yukawa type. 
\REV{Moreover, away from $x_1$ (where the center of $e_1$ is located), local $L^{\infty}$ bounds can be obtained from the $L^2$ estimate combined with elliptic regularity estimates for the finite difference equation (see e.g., \cite{ThomeeWestergren:68}).} 
In summary, this establishes the exponential moment bound for $g$
uniform in $L$ and the discretization mesh size, thus, the Green's
function decays exponentially along the off-diagonal direction.

\section{Pseudo-spectral method and mollified pseudo-spectral
method}\label{sec:ps}

In this section we consider the pseudo-spectral type discretization.
When the potential function $V$ is smooth, pseudo-spectral
discretization is widely used in scientific and engineering
computations. This is because pseudo-spectral type discretization gives
rise to much more accurate solution than low order finite difference
type discretization with the same number of degrees of freedom. 

In pseudo-spectral type discretization, corresponding to the discrete
lattice $\mc{X}$ we define the
Fourier grid $\mc{K}=\left\{ n\Delta k \mid n=-\frac{N}{2}+1,\ldots,
\frac{N}{2} \right\}$. Here $\Delta k=\frac{2\pi}{L}$, 
the \emph{edge} of the Fourier grid is defined to be 
\begin{equation}
  k_{c}:=\frac{N}{2}\Delta k=\frac{\pi N}{L} = \frac{\pi}{\Delta x}.
\end{equation}
Note that $k_c \ge \pi$ due to the assumption $\Delta x \le 1$.

For a lattice function $f: \mc{X} \to \RR$, its discrete
Fourier transform is defined as
\[
\wh{f}_{k} = \Delta x \sum_{x\in \mc{X}} e^{-\I k x} f(x), \quad k \in \mc{K}.
\] 
The corresponding inverse discrete Fourier transform is 
\[
f(x) = \frac{1}{L} \sum_{k \in \mc{K}} e^{\I k x} \wh{f}_k, \quad
x\in\mc{X}.
\]
Here the normalization factor is chosen so that when the grid spacing
$\Delta x\to 0$, the discrete Fourier transform and inverse Fourier transform
converges to the continuous Fourier transform and inverse Fourier
transform, respectively.  

Similar to Eq.~\eqref{eqn:normL2} and~\eqref{eqn:normLinf}, in Fourier
space, the discrete $L^{2}(\mc{K})$ norm and $L^{\infty}(\mc{K})$ norm for
$\{\wh{f}_{k}\}$ is given as 
\begin{equation}
  \norm{\wh{f}}^2_{L^{2}(\mc{K})} = \Delta k \sum_{k\in \mc{K}} \abs{\wh{f}_{k}}^2, \quad
  \norm{\wh{f}}_{L^{\infty}(\mc{K})} = \max_{k\in \mc{K}} \abs{\wh{f}_{k}},
  \label{eqn:normfourier}
\end{equation}
respectively. 
Again for simplicity of the notation, we will use $\norm{\wh{f}}_{2}$ and
$\norm{\wh{f}}_{\infty}$ interchangeably with $\norm{\wh{f}}_{L^{2}(\mc{K})}$ and
$\norm{\wh{f}}_{L^{\infty}(\mc{K})}$, respectively, unless otherwise
clarified.

Under this choice of normalization, the discrete Parseval's identity
reads
\begin{equation}\label{eq:parseval}
  \begin{aligned}
    \norm{\wh{f}}^2_2 &= \Delta k \sum_{k \in \mc{K}} \abs{\wh{f}_k}^2 
    = \Delta k (\Delta x)^2 \sum_{k \in \mc{K}} \sum_{x, x'\in\mc{X}} e^{-\I
    k(x-x')} f(x) f^{\ast}(x') \\
    & = \Delta k N (\Delta x)^2 \sum_{x\in\mc{X}} \abs{f(x)}^2 = 2\pi
    \norm{f}_2^2.
  \end{aligned}
\end{equation}

We define the Fourier restriction operator $\mc{R}_{N}:L^{2}(\Omega)\to
L^{2}(\mc{X})$ as
\[
(\mc{R}_{N}g(\cdot))(x) = g(x), \quad x\in \mc{X}.
\]
Similarly the Fourier interpolation operator $\mc{I}_{N}:
L^{2}(\mc{X})\to L^{2}(\Omega)$ is defined as
\[
[\mc{I}_{N}f](x) = \frac{1}{L}\sum_{k\in \mc{K}} \wh{f}_{k} e^{\I k x},
\quad x\in \Omega.
\]
Using the Fourier restriction and interpolation operator, the
Laplacian operator in the pseudo-spectral discretization 
becomes $\mc{R}_{N} \Delta \mc{I}_{N}$. For simplicity we consider the
case in the absence of the external potential i.e. $V(x)=0$, and
$\lambda=-1$. In this case, the pseudo-spectral discretization is
equivalent to the spectral discretization, and Eq.~\eqref{eqn:problem} becomes 
\begin{equation}
  (1 - \mc{R}_{N} \Delta \mc{I}_{N}) G = -\frac{1}{\Delta x} I.  
  \label{eqn:psdiscretize}
\end{equation}
Again the prefactor $1/\Delta x$ on the right hand side of
Eq.~\eqref{eqn:psdiscretize}
reflects the normalization of the discrete Dirac $\delta-$distribution.
Since $\mc{R}_{N} \Delta \mc{I}_{N}$ is translational invariant, without
loss of generality we only consider the first column of $G$, denoted by
$g$. Then
\begin{equation}
  (1 - \mc{R}_{N} \Delta \mc{I}_{N}) g = -\frac{1}{\Delta x} e_{1},
  \label{eqn:psdiscrete2}
\end{equation}
where $e_{1}=(1,0,\ldots,0)^{T}$. 
Direct computation shows that
\[
\wh{g}_{k} = -\frac{1}{1+k^2}, \quad k\in \mc{K}.
\]

Below we would like to utilize the discrete version of the relation
between the regularity of the Fourier space and the decay in the real
space. This allows us to obtain the decay properties of $g$ by
estimating the norm of $\wh{g}$ and its discrete derivatives.
Let us first note an elementary calculus lemma. 
\begin{lemma}
  Let $\abs{x}\in \left[
  0,\frac{L}{2} \right]$. Then
  \[
    \frac{\abs{e^{\I \Delta k x}-1}}{\Delta k} \ge \frac{2
    \abs{x}}{\pi}. 
  \]
  \label{lem:xbound}
\end{lemma}
\begin{proof}
  Note that $\frac{\abs{\Delta k x}}{2} \le \frac{\pi}{2}$, and
  $\sin y \ge \frac{2}{\pi} y$ for $0\le y\le \frac{\pi}{2}$, we have
  \[
  \frac{\abs{e^{\I \Delta k x}-1}}{\Delta k}
  = \frac{2\abs{\sin\left( \frac{\Delta k x}{2} \right)}}{\Delta k}
  \ge
  \frac{2 \abs{x}}{\pi}.
  \]
\end{proof}

We define the difference operator $\mc{D}$ acting on a vector
$\wh{f}$ in the Fourier domain as
\begin{equation}
  (\mc{D} \wh{f})_{k} =
    \frac{\wh{f}_{k}-\wh{f}_{k-\Delta k}}{\Delta k}, \quad k\in \mc{K}.
  \label{eqn:Doperator}
\end{equation}
Eq.~\eqref{eqn:Doperator} is interpreted in the periodic sense, i.e.
for $k_{-\frac{N}{2}+1}$, $k_{-\frac{N}{2}+1}-\Delta k \equiv 
k_{\frac{N}{2}}$. Proposition~\ref{thm:discretedecay} characterizes the
decay property of $g$ in terms of the first order difference of
$\wh{g}$.

\begin{prop}
  Define 
  \[
  d(x,0) = \begin{cases}
    x, & x\in [0,L/2),\\
    L-x, & x\in [L/2,L)
  \end{cases}
  \]
  then for any $g\in L^{2}(\mc{X})$,
  \[
  \norm{d(\cdot,0) g}_{2} \leq %\frac{1}{2 \sqrt{2\pi}}
  \frac{\sqrt{\pi}}{2\sqrt{2}}
  \norm{\mc{D}\wh{g}}_{2}.
  \]
  \label{thm:discretedecay}
\end{prop}
\begin{proof}
  For any $x\in\mc{X}$, since $0\le d(x,0)\le \frac{L}{2}$, using
  Lemma~\ref{lem:xbound}, 
  \[
  \abs{d(x,0) g(x)} \le \frac{\pi}{2} \frac{\abs{(e^{\I \Delta k
  x}-1)g(x)}}{\Delta k}.
  \]
  Since
  \[
  g(x) = \frac{1}{L} \sum_{k\in \mc{K}} e^{\I k x} \wh{g}_{k},
  \]
  we have
  \[
  (e^{\I \Delta k x}-1) g(x) = \frac{1}{L} \sum_{k\in \mc{K}} e^{\I (k+\Delta
    k) x} \wh{g}_{k} - \frac{1}{L} \sum_{k\in \mc{K}} e^{\I k x}
    \wh{g}_{k}.
  \]
  Rearrange the terms and use the definition of
  Eq.~\eqref{eqn:Doperator}, we have
  \[
  \frac{e^{\I \Delta k x}-1}{\Delta k} g(x) = 
  -\frac{1}{L} \sum_{k\in \mc{K}} e^{\I k x}  (D\wh{g})_{k}.
  \]
  Summing up over all $x\in\mc{X}$, we obtain
  \[
  \Delta x\sum_{x\in\mc{X}}\abs{\frac{e^{\I \Delta k x}-1}{\Delta k}
  g(x)}^2 
  = \frac{1}{L}
  \sum_{k}\abs{(D\wh{g})_{k}}^2,
  \]
  and therefore
  \[
  \norm{d(\cdot,0) g}_{2} \le \frac{\sqrt{\pi}}{2\sqrt{2}}
  \norm{\mc{D}\wh{g}}_{2}.
  \]
\end{proof}

Applying Proposition~\ref{thm:discretedecay} repeatedly for $m$ times, we
have
\begin{corollary}
  \label{cor:discretedecaym}
  For any $g\in L^{2}(\mc{X})$ and positive integer $m$,
  \[
  \norm{d(\cdot,0)^{m} g}_{2} \le \left(\frac{\pi}{2}\right)^{m} \frac{1}{\sqrt{2\pi}}
  \norm{\mc{D}^{(m)}\wh{g}}_{2}.
  \]
\end{corollary}
\begin{proof}
  The proof follows from the identity that for any $x\in\mc{X}$
  \[
  \frac{\left(e^{\I \Delta k x}-1\right)^{m}}{\Delta k} g(x) = 
  \frac{(-1)^{m}}{L}\sum_{k\in \mc{K}} e^{\I k x}  (D^{(m)}\wh{g})_{k},
  \]
  and a similar calculation as in Proposition~\ref{thm:discretedecay}. 
\end{proof}

Corollary~\ref{cor:discretedecaym} suggests that in order to obtain high
order polynomial decay rate, we need to control the high order
derivatives of $\wh{g}$. However, the difficulty associated with the
pseudo-spectral method is that the discrete Laplacian in the Fourier
space is $k^2$ and is not smooth at the edge of the Fourier grid
$k=\pm k_{c}$.
Numerical results in section~\ref{sec:numer} indicate that the
off-diagonal elements of the discretized Green's function from
pseudo-spectral discretization indeed decay slowly in the asymptotic
sense.

Below we demonstrate that it is possible to mollify the pseudo-spectral
scheme which smears the discontinuity near the edge of the Fourier
grid $\pm k_{c}$, and the
resulting discretized Green's function decays faster than
$d(x,0)^{-M}$ along the off-diagonal direction, where $M\sim
\Or(N)$. As a result, as the system size $L$ and hence $N$ increases,
the decay along the off-diagonal direction is super-algebraic, i.e. faster than any
polynomial of $d(x,0)$. 

\medskip 

For pseudo-spectral discretization, the following discrete version of
the Leibniz rule plays an important role.
\begin{lemma}
  For any $\wh{f},\wh{g}\in L^{2}(\mc{K})$, and $k\in\mc{K}$,
  \[
  (\mc{D}[\wh{f}\wh{g}])_{k} = (\mc{D}\wh{f})_{k} \wh{g}_{k-\Delta k} +
  \wh{f}_{k}(\mc{D}\wh{g})_{k}, 
  \]
  and 
  \[
   (\mc{D}[\wh{f}\wh{g}])_{k} = (\mc{D}\wh{f})_{k} \wh{g}_{k} +
   \wh{f}_{k-\Delta k}(\mc{D}\wh{g})_{k}.
  \]
  \label{lem:proddiff}
\end{lemma}
\begin{proof} The proof is elementary. 
  \[
  \begin{split}
  (\mc{D}[\wh{f}\wh{g}])_{k} &= \frac{1}{\Delta k} \left(\wh{f}_{k}
  \wh{g}_{k} - \wh{f}_{k-\Delta k}\wh{g}_{k-\Delta k}\right)\\
  &= \frac{1}{\Delta k}\left( (\wh{f}_{k}-\wh{f}_{k-\Delta k})
  \wh{g}_{k - \Delta k} +
  \wh{f}_{k} (\wh{g}_{k}-\wh{g}_{k-\Delta k}) \right)\\
  & =  (\mc{D}\wh{f})_{k} \wh{g}_{k-\Delta k} + \wh{f}_{k}(\mc{D}\wh{g})_{k}.
  \end{split}
  \]
  The second equality follows by switching the role of $f$ and $g$. 
\end{proof}

Let us  introduce a smooth cut-off
function $\wh{\theta}(k)\in C^{\infty}(\mathbb{R})$ which satisfies
\begin{equation}
\wh{\theta}(k) = 
\begin{cases}
  1, & \abs{k} \le \frac12 k_{c},\\
  0, & \abs{k} \ge \frac34 k_{c},
\end{cases}
  \label{eqn:thetacond}
\end{equation}
and $0\le \wh{\theta}(k)\le 1$. 
For example, we can choose $\wh{\theta}$ to be a characteristic function
$\wh{\theta}_{0}(k):=\boldsymbol{1}_{\abs{k} \leq \frac{5}{8} k_c}$
convolved with a ``bump'' function $\wh{\varphi}(k)$, i.e.
\begin{equation}
  \wh{\theta}(k) = \int \wh{\varphi}(k-k') \wh{\theta}_{0}(k')\ud k',
  \label{eqn:theta}
\end{equation}
and
\begin{equation}
  \wh{\varphi}(k)=\begin{cases}
    Z \exp\left( -\frac{\sigma^2 k_c^2}{\sigma^2 k_c^2-k^2} \right), &
    \abs{k}<\sigma k_{c},\\
    0,& \mathrm{otherwise}.
  \end{cases}
  \label{eqn:bump}
\end{equation}
Here $Z$ is a normalization constant chosen so that $\int
\wh{\varphi}(k)\ud k = 1$, and we choose $\sigma=\frac18$. An
example of the mollification function $\wh{\theta}(k)$ is given in
Fig.~\ref{fig:theta}. 

%We further require that the cut-off function satisfies 
%\begin{equation}
%  \norm{ \frac{\ud^{m} \wh{\theta} }{\ud k^m} }_{\infty} \lesssim 2^m, 
%\end{equation}
%as a corollary, we have 
%\begin{equation}\label{eq:boundDtheta}
%  \norm{ \mc{D}^{(m)} \wh{\theta} }_{\infty} \lesssim 2^m. 
%\end{equation}
%Such a cut-off function may be obtained by the characteristic function
%$\boldsymbol{1}_{\abs{k} \leq \frac{5}{8} k_c}$ mollified by the
%standard bump function $\varphi = Z^{-1} e^{-1/(1 - \abs{k}^2)}\boldsymbol{1}_{\abs{k}\leq 1}$,
%where $Z$ is a normalization constant such that $\int \varphi =
%1$.
%\LL{Don't see how to use this bump function to do the job that
%smoothly connects $1$ and $0$. In the new
%code I used another function instead as cutoff function that takes the
%form 
%\[
%\frac{e^{-1/x}}{e^{-1/(1-x)}+e^{-1/x}}.
%\]
%}
%It
%meets all the requirements provided that $k_c \geq 8$ so that the mollified function satisfies \eqref{eqn:thetacond}. 

%\medskip 

To remove the singularity of the symbol $k^2$ near the edge of the
Fourier grid $\mc{K}$, we introduce a mollified kernel of Laplacian operator in
the Fourier domain as
\begin{equation}
  \wh{h}(k) = \wh{\theta}(k) (k^2 - k_{c}^2)  + k_{c}^2, \quad k\in\RR.
  \label{eqn:hmodlap}
\end{equation}
It is easy to verify that $\wh{\theta}\in C^{\infty}(\mathbb{R})$, and
then $\wh{h}(k)\in C^{\infty}(\mathbb{R})$.  However, since the bump
function $\wh{\varphi}$ is only $C^{\infty}(\mathbb{R})$ but not real
analytic at $k=\pm \sigma k_{c}$, its Fourier transform is known to
decay super-algebraically and
sub-exponentially~\cite{Johnson2015}. Hence exponential decay of the
off-diagonal direction of the Green's function cannot be
expected. Below we prove that for such choice of the mollified
pseudo-spectral scheme, the off-diagonal direction of the Green's
function decays super-algebraically. To this end we follow
Corollary~\ref{cor:discretedecaym} and need to bound the high order
difference operators applied to $\wh{g}$. Our current proof does not
give sub-exponential bound, which is an interesting future direction.

We assume that for each integer $m\ge 0$, there exists constants
$C_{\theta,m}$ independent of $k_{c}$ so that (recall that $k_c
\geq \pi$ and hence $\sigma k_c$ is bounded from below by $\pi / 8$)
\[
\norm{ \frac{\ud^{m} \wh{\theta} }{\ud k^m} }_{\infty} \le C_{\theta,m}, 
\]
and hence
\begin{equation}\label{eq:boundDtheta}
  \norm{ \mc{D}^{(m)} \wh{\theta} }_{\infty} \le C_{\theta,m}.
\end{equation}
%\LL{Below we have mixed usage of universal constant $C_{m}$ and
%$\lesssim$. Not sure whether this creates confusion or not. Perhaps just
%keep one of two approaches for getting rid of constants?}

We further have the
following lemma for controlling the derivative of $\wh{h}_{k}\equiv
\wh{h}(k)$.
\begin{lemma}\label{lem:hbound1}
  % There exists a constant $C_{h,0},C_{h,1}$ independent of $L$ such that
  % \[
  % \norm{\wh{h}}_{\infty} \le C_{h,0}, \quad
  % \norm{\mc{D}\wh{h}}_{\infty} \le C_{h,1}. \quad 
  % \norm{\frac{\mc{D}\wh{h}}{1 + \wh{h}}} 
  % \]

  Assume $1\le m \leq M = \frac{N}{16}$,
%  , $N$ is a multiple of $16$, \jl{shall we remove the assumption that $N$ is a multiple of $16$? I think it works as long as $M \leq N/16$ even if $N/16$ is a fractional number}
  $\Delta x \leq 1$ and $L\geq 1$. Then
  there exist constants $C_{h,m}$ independent of $k_{c},L$ such that
  \begin{equation*}
    \norm{\frac{\mc{D}^{(m)}\wh{h}}{1 + \wh{h}}}_{\infty} 
    \le C_{h,m}.
  \end{equation*}
%  In particular, the bounds are independent of $L$ and $k_c$.
\end{lemma}
\begin{proof}
  Use Lemma~\ref{lem:proddiff},
  \[
  (\mc{D}\wh{h})_{k} = (\mc{D}\wh{\theta})_{k} (k^2-k_c^2) +
  \wh{\theta}_{k-\Delta k} (\mc{D}[k^2])_{k} = 
  (\mc{D}\wh{\theta})_{k} (k^2-k_c^2) +
  \wh{\theta}_{k-\Delta k} (2k-\Delta k) , \quad k\in \mc{K}, 
  \]
  where the last equality used that $k^2$ is interpreted in the
  periodic sense.  Thus,
  \begin{equation}\label{eq:Dh}
    \frac{(\mc{D} \wh{h})_k}{1 + \wh{h}_k}
    = \frac{k^2-k_c^2}{1 + \wh{h}_k}   (\mc{D}\wh{\theta})_{k} +
    \frac{2k-\Delta k}{1 + \wh{h}_k}  \wh{\theta}_{k-\Delta k}. 
  \end{equation}
  To bound the right hand side, we use 
  \begin{equation}\label{eq:estimate1}
    \abs{\frac{2k-\Delta k}{1 + \wh{h}_k}} = 
    \abs{\frac{2k - \Delta k}{ 1 + \wh{\theta}(k) (k^2 - k_{c}^2)  + k_{c}^2}}
    \leq \abs{\frac{2k - \Delta k}{1 + k^2}} \le 1 + \Delta k \le 10.
  \end{equation}
  Here $L \ge 1$ and hence $\Delta k\le 2\pi$.
  For the first term of the right hand side of \eqref{eq:Dh}, note that \begin{equation}
    (\mc{D} \wh{\theta})_k = 0, \qquad \text{if }\abs{k} \leq \frac{1}{2} k_c - \Delta k.
  \end{equation}
  Moreover, for $\abs{k} \geq \frac{1}{4} k_c$, we have 
  \begin{equation}\label{eq:estimate2}
    \abs{\frac{k^2-k_c^2}{1 + \wh{h}_k}} = 
    \abs{\frac{k^2-k_c^2}{ 1 + \wh{\theta}(k) (k^2 - k_{c}^2)  + k_{c}^2}} 
    \leq \frac{2 k_c^2} {1 + k^2} \le 32,
  \end{equation}
  where the last inequality uses the lower bound of $k$ as
  assumed. Therefore, we arrive at
  \begin{equation}
    \norm{\frac{\mc{D} \wh{h}}{1 + \wh{h}}}_{\infty} \le 
    \left(32 \norm{\mc{D} \wh{\theta}}_{\infty} + 10 \norm{\wh{\theta}}_{\infty}
    \right) := C_{h,1}.
  \end{equation}
  where we have used \eqref{eq:boundDtheta} in the last inequality. 

  Controlling higher derivatives of $\wh{h}$ is similar. Apply
  $\mc{D}$ and Lemma~\ref{lem:proddiff} for $m$ times on both sides of
  Eq.~\eqref{eqn:hmodlap}, we obtain
  \begin{multline}
    (\mc{D}^{(m)}\wh{h})_{k} = (\mc{D}^{(m)}\wh{\theta})_{k}
    (k^2-k_c^2) + {m \choose 1} (\mc{D}^{(m-1)}\wh{\theta})_{k-\Delta
      k} (2k-\Delta k) \\
    + {m \choose 2} 2 (\mc{D}^{(m-2)}\wh{\theta})_{k-2\Delta k},\quad
    k\in \mc{K}.
    \label{eqn:hdiffm}
  \end{multline}
  The reason why $\mc{D}^{2}[k^2]_{k}$ can be replaced by $2$ is because
  $(\mc{D}^{(m-2)}\wh{\theta})_{k-2\Delta k}$ vanishes at the boundary of
  $\mc{K}$. Similarly the right hand side of the equation
  Eq.~\eqref{eqn:hdiffm} stops at the term $(\mc{D}^{(m-2)}\wh{\theta})$ is
  because when $3\le m\le M$, 
  let 
  \[
  \mc{K}_{m} = \left\{ \left(-\frac{N}{2}+m\right)\Delta k,\ldots,
  \left(\frac{N}{2}-m+1\right) \Delta k \right\},
  \]
  we have
  \[
  (\mc{D}^{(m)} [k^2])_{k} = 0,\quad k\in \mc{K}_{m}.
  \]
  On the other hand,  since
  \[
  \wh{\theta}_{k} = 0,\quad k\in
  \mc{K}\backslash\mc{K}_{\frac{N}{8}},
  \]
  then for all $3\le m\le M$,
  \[
  (\mc{D}^{(m)} \wh{\theta})_{k} = 0,\quad k\in \mc{K}\backslash
  \mc{K}_{\frac{N}{8}-m}.
  \]
  Since $2m \le \frac{N}{8}$, all terms of the form 
  \[
  (\mc{D}^{(m-n)}\wh{\theta})_{k-n\Delta k} (\mc{D}^{(n)}
  [k^2])_{k} = 0, \quad k\in \mc{K}, \quad 3\le n\le m.
  \]
  Hence
  \begin{multline}
    \frac{(\mc{D}^{(m)}\wh{h})_{k}}{1 + \wh{h}_k} = \frac{k^2-k_c^2}{
      1 + \wh{h}_k} (\mc{D}^{(m)}\wh{\theta})_{k} + 
      \frac{2k-\Delta k}{1 + \wh{h}_k} {m \choose 1} (\mc{D}^{(m-1)}\wh{\theta})_{k-\Delta
      k}  \\
    + \frac{2}{1 + \wh{h}_k} {m \choose 2} 
    (\mc{D}^{(m-2)}\wh{\theta})_{k-2\Delta k},\quad k\in \mc{K}.
  \end{multline}
  Using \eqref{eq:estimate1}, \eqref{eq:estimate2}, and the fact that
  $(\mc{D}^{(m)}\wh{\theta})_k = 0$ for $m \leq \frac{N}{16}$ and
  $\abs{k} \leq \frac{1}{4} k_c$, we arrive at 
%    \norm{\frac{(\mc{D}^{(m)}\wh{h})_{k}}{1 + \wh{h}_k}}_{\infty} 
%    \lesssim \norm{\mc{D}^{(m)} \wh{\theta}}_{\infty} 
%    + m \norm{\mc{D}^{(m-1)}\wh{\theta}}_{\infty} 
%    + m^2 \norm{\mc{D}^{(m-2)}\wh{\theta}}_{\infty} \\
%    \stackrel{\eqref{eq:boundDtheta}}{\lesssim} 2^m + m 2^{m-1} + m^2 2^{m-2} 
%    \lesssim m^2 2^{m-2}. 
  \begin{multline*}
    \norm{\frac{\mc{D}^{(m)}\wh{h}}{1 + \wh{h}}}_{\infty} 
    \le (32\norm{\mc{D}^{(m)} \wh{\theta}}_{\infty} 
    + 2m \norm{\mc{D}^{(m-1)}\wh{\theta}}_{\infty} 
    + m(m-1) \norm{\mc{D}^{(m-2)}\wh{\theta}}_{\infty}) := C_{h,m}. 
  \end{multline*}
\end{proof}

\medskip 

For the mollified pseudo-spectral discretization, we replace $k^2$ by
$\wh{h}_{k}$ for all $k\in \mc{K}$. We study below the
decay properties of $g$ with Fourier transform
denoted by $\wh{g}$. From Eq.~\eqref{eqn:psdiscrete2},
$\wh{g}$ satisfies 
\begin{equation}
  (1 + \wh{h}_{k}) \wh{g}_{k} = -1, \quad k\in \mc{K}.
  \label{eqn:modlap}
\end{equation}
Applying $\mc{D}$ to both sides of Eq.~\eqref{eqn:modlap} and use
Lemma~\ref{lem:proddiff}, we have
\begin{equation}
  (\mc{D}\wh{h})_{k} \wh{g}_{k-\Delta k} + (1+\wh{h}_{k})
  (\mc{D} \wh{g})_{k} = 0,\quad k\in \mc{K}.
  \label{eqn:modlapdif1}
\end{equation}

%In order to bound the derivatives of $\wh{g}$ up to a sufficiently high
%order derivatives $\norm{ \mc{D}^{(m)} \wh{g}}_2$, we use Abel's
%binomial theorem (see e.g., \cite[Page 18]{Riordan:68}), which generalizes
%Newton's binomial theorem.  For completeness, a proof of Abel's binomial
%theorem is given in the Appendix of this manuscript.

\begin{theorem}\label{thm:Dmboundfinish}
  Let $g\in L^2(\mc{X})$ be the inverse Fourier transform of $\wh{g}$
  defined in Eq.~\eqref{eqn:modlap}.
  Assume $N\ge 32$,
  $\Delta x\leq 1$ and $L\geq 1$.
  Then there exist constants $C_{g,m}$ independent of $L$ and $k_{c}$ such
  that for all $0\le m\le M=\frac{N}{16}$, 
  \[
  \norm{d(\cdot,0)^{m} g}_{2} \le C_{g,m}.
  \]
\end{theorem}

%\LL{Still do not think this $\sqrt{\Delta x}$ makes sense: cann't we
%just compute $\norm{\wh{g}}_{2}$ and show that this is a number
%independent of $\Delta x$? Then the inherited dependence on
%$\sqrt{\Delta x}$ on higher order derivatives will also be gone.} 

% \begin{theorem}
%   Let $N\ge 16$, then there exists a constant $C_{1}$ independent of
%   $L$ such that 
%   \[
%   \norm{d(x_{j},0) g_{j}}_{2} \le \frac{\pi}{2}
%   C_{1},\quad j=1,\ldots,N.
%   \]
%   \label{thm:D1boundfinish}
% \end{theorem}

\begin{proof}
  First,
  \[
  \begin{split}
  \norm{\wh{g}}_{2}^2 =& \Delta k \sum_{k\in \mc{K}}
  \frac{1}{(1+\wh{h}_{k})^2} \le \Delta k \sum_{k\in \mc{K}}
  \frac{1}{(1+k^2)^2} \\
  \le& \int_{-\infty}^{\infty}
  \frac{1}{(1+k^2)^2} \ud k + \Delta k < \int_{-\infty}^{\infty}
  \frac{1}{(1+k^2)^2} \ud k + 2\pi := C_{g,0}^{2}.
  \end{split}
  \]
  Here we used that $\wh{h}_{k}\ge k^2$ for $k\in \mc{K}$, and
  $\Delta k = 2\pi/L \leq 2\pi$.
%  \[
%  \norm{\wh{g}}_{2} \le \frac{1}{\abs{1}} \sqrt{N \Delta k}  =
%  \frac{1}{\abs{1}} \sqrt{\frac{2\pi}{\Delta
%  x}}:=\frac{C_{0}}{\sqrt{\Delta x}}. 
%  \] 
  From Eq.~\eqref{eqn:modlapdif1}, we have
  \begin{equation*}
    \abs{(\mc{D}\wh{g})_{k}} = \abs{\frac{(\mc{D}\wh{h})_{k}
  \wh{g}_{k-\Delta k}}{1+\wh{h}_{k}}}, 
  \end{equation*}
  thus
  \begin{equation*}
    \norm{ \mc{D} \wh{g} }_2 \leq  \norm{\frac{\mc{D}
    \wh{h}}{1 + \wh{h}}}_{\infty} \norm{\wh{g}}_2 \le
    C_{h,1}\norm{\wh{g}}_2 := C_{g,2},
  \end{equation*}
  where the last inequality uses Lemma~\ref{lem:hbound1}. 

  Apply $\mc{D}$ and
  Lemma~\ref{lem:proddiff} for $m$ times on both sides of
  Eq.~\eqref{eqn:modlap}, we have
  \[
  \sum_{n=0}^{m-1} {m \choose n}(\mc{D}^{(m-n)}\wh{h})_{k}
  (\mc{D}^{(n)}\wh{g})_{k-n\Delta k} + (1 + \wh{h}_{k})
  (\mc{D}^{(m)}\wh{g})_{k} = 0,\quad k\in \mc{K}.
  \]
  Hence
  \begin{equation}\label{eq:Dmfbound}
    \begin{aligned}
      \norm{ \mc{D}^{(m)} \wh{g} }_2 & \leq \sum_{n=0}^{m-1} {m
        \choose n} 
        \norm{\frac{\mc{D}^{(m-n)}\wh{h}}{1 + \wh{h}}}_{\infty}
      \norm{ \mc{D}^{(n)}\wh{g} }_2  \\
      & \le \sum_{n=0}^{m-1} {m
      \choose n} C_{h,m-n} C_{g,n} := C_{g,m}.
    \end{aligned}
  \end{equation}
\end{proof}

The case with general value of $\lambda$ in the resolvent set of
$H$, and general potential function $V(x)$ is very similar. We denote by 
$V:\mc{X}\to \RR$ the value of the potential function $V(x)$ evaluated
on the lattice $\mc{X}$. The Fourier transform of $V$ is denoted by $\wh{V}$.
Define  the matrix in the Fourier space
\begin{equation}
  \wh{H}_{kl} = \wh{h}_{k} \delta_{kl} + \frac{1}{L}\wh{V}_{k-l}, \quad
  k,l\in \mc{K}.
  \label{eqn:fourierH}
\end{equation}
Here $\delta_{kl}$ is the Kronecker-$\delta$ symbol.
Then the mollified pseudo-spectral discretization of Eq.~\eqref{eqn:problem},
represented in the Fourier space becomes
\begin{equation}
  \lambda \wh{g}_{k} - \sum_{l\in \mc{K}} \wh{H}_{kl} \wh{g}_{l} = 1, \quad k\in
  \mc{K}.
  \label{eqn:modVeq}
\end{equation}
When repeatedly applying $\mc{D}$ to both sides of
Eq.~\eqref{eqn:modVeq}, Lemma~\ref{lem:fVdiff} indicates that all the
differences can be applied to $\wh{g}$.
\begin{lemma}\label{lem:fVdiff}
  \[
  \left[\mc{D}\left(\sum_{l\in \mc{K}} \wh{V}_{\cdot-l}
  \wh{g}_{l}\right)\right]_{k} = \sum_{l\in
  \mc{K}} \wh{V}_{k-l} (\mc{D}\wh{g})_{l}, \quad k\in \mc{K}.
  \]
\end{lemma}
\begin{proof}
  \[
  \begin{split}
  \left[\mc{D}\left(\sum_{l\in \mc{K}} \wh{V}_{\cdot-l}
  \wh{g}_{l}\right)\right]_{k} =& 
  \frac{1}{\Delta k}\left\{\sum_{l\in \mc{K}} \wh{V}_{k-l} \wh{g}_{l} - 
  \sum_{l\in \mc{K}} \wh{V}_{k-\Delta k-l} \wh{g}_{l}\right\} \\
  =&
  \frac{1}{\Delta k}\left\{\sum_{l\in \mc{K}} \wh{V}_{k-l} \wh{g}_{l} - 
  \sum_{l\in \mc{K}} \wh{V}_{k-l} \wh{g}_{l-\Delta k}\right\} \\
  =&
  \sum_{l\in
  \mc{K}} \wh{V}_{k-l} (\mc{D}\wh{g})_{l}.
  \end{split}
  \]
\end{proof}

Now we prove the decay properties of discretized Green's functions for
the mollified pseudo-spectral discretization in
Theorem~\ref{thm:Dmboundgeneral}.
\begin{theorem}
  Let $g\in L^2(\mc{X})$ be the inverse Fourier transform of $\wh{g}$
  defined in Eq.~\eqref{eqn:modVeq}.
  Assume $N\ge 32$, $\Delta x\leq 1$, $L\geq 1$,
  and $V \in L^{\infty}(\mc{X})$.
  Assume the discretized Green's function $\wh{G}=(\lambda-\wh{H})^{-1}$
  has bounded matrix 2-norm
  $\norm{G}_{\mc{L}(L^{2}(\mc{K}))}$.  
  Then there exists constants $C_{V,g,m}$ independent of $L,k_{c}$ such
  that for all $0\le m\le M =
  \frac{N}{16}$, 
  \[ 
  \norm{d(\cdot,0)^{m} g}_{2} \le C_{V,g,m}.
  \]
  \label{thm:Dmboundgeneral}
\end{theorem}
%\begin{remark}
%  The factor $\frac{1}{L}$ in the definition of $A$ is natural due to
%  our scaling of the Fourier transform, cf.~\eqref{eq:Vf}. 
%\end{remark}

%\LL{Still do not think this $\sqrt{\Delta x}$ makes sense: If we can
%show $\sqrt{\Delta x}$ should not be there in Theorem 4.6, can we assume
%$\norm{\wh{g}}_{2}$ is also a bounded here, just as assuming that
%$G_{0}$ has uniformly bounded norm?  } 

\begin{proof}
  The proof is similar to the proof of
  Theorem~\ref{thm:Dmboundfinish}, and we will only focus on the new
  argument for treating general $\lambda$ and $V$.

  Note that for $k\in \mc{K}$,
  \begin{equation*}
    \begin{split}
    \frac{1}{L}\sum_{l\in\mc{K}} \wh{V}_{k-l} \wh{g}_l = &
    \frac{(\Delta x)^2}{L} \sum_{l\in\mc{K}} \sum_{x,x'\in\mc{X}} e^{-\I
    (k-l) x} V(x) e^{-\I l x'} g(x') \\
    = &\frac{(\Delta x)^2 N}{L} \sum_{x\in\mc{X}} e^{-\I k x} V(x) g(x) = \wh{(Vg)}_k. 
    \end{split}
  \end{equation*}
  Thus, using the Parseval's identity \eqref{eq:parseval}, we get
  \begin{equation}\label{eq:Vinfty}
    \norm{\frac{1}{L}\sum_{l\in\mc{K}} \wh{V}_{\cdot-l}
    \wh{g}_l}_{2} = \sqrt{2\pi}  \norm{ V g }_2 
    \leq \sqrt{2\pi}  \norm{V}_{\infty} \norm{g}_2.
  \end{equation}
  Let us introduce the notation $(\wh{M}_V)_{kl} := \frac{1}{L}
  \wh{V}_{k-l}$, and simply denote by $\wh{h}$ the diagonal matrix
  with diagonal entries being $\wh{h}_{k},k\in \mc{K}$. Then the above estimate shows that the matrix $2$-norm
  $\lVert\wh{M}_V\rVert_{\mc{L}(L^2(\mc{K}))}$ is bounded by
  $\sqrt{2\pi}\norm{V}_{\infty}$.  Notice that
  \begin{multline}\label{eq:GV}
    \norm{\wh{G} (1 + \wh{h})}_{\mc{L}(L^2(\mc{K}))} = \norm{\wh{G} (
    1+\lambda- \wh{M}_V-\lambda+\wh{H})}_{\mc{L}(L^2(\mc{K}))} \\
    \leq 1 + \norm{\wh{G}}_{\mc{L}(L^2(\mc{K}))}
    (\abs{1+\lambda}+\sqrt{2\pi}\norm{V}_{\infty}):=C_{\wh{G}},
  \end{multline}
  where the last inequality follows from \eqref{eq:Vinfty}. Hence,
  $\norm{\wh{G} (1 + \wh{h})}_{\mc{L}(L^2(\mc{K})}$ is bounded by the
  constant $C_{\wh{G}}$. 

  %  Eq.~\eqref{eqn:modVeq} implies that $\wh{g}_l  = \sum_k  \wh{G}_{lk}$,
  %  and use Theorem~\ref{thm:Dmboundfinish}, we have
  From Eq.~\eqref{eqn:modVeq} and use Theorem~\ref{thm:Dmboundfinish}, we have
  \[
  \norm{\wh{g}}_{2} \le \norm{\wh{G}(1+\wh{h})}_{\mc{L}(L^2(\mc{K}))}
  \norm{(1+\wh{h})^{-1}}_{2} \le C_{\wh{G}} C_{g,0} := C_{V,g,0}.
  \]

  For $m=1$, apply $\mc{D}$ on both side of Eq.~\eqref{eqn:modVeq}, and
  \[
  (\mc{D}\wh{h})_{k} \wh{g}_{k-\Delta k} + \sum_{l\in \mc{K}}\left[(\lambda -
  \wh{h}_{k})\delta_{kl} -
  \frac{1}{L}\wh{V}_{k-l}\right] (\mc{D}\wh{g})_{l} = 0.
  \]
  Hence
  \begin{equation*}
    (\mc{D}\wh{g})_l = \sum_{k \in \mc{K}} (\wh{G})_{lk}
    (1+\wh{h}_k) \frac{(\mc{D}\wh{h})_{k}}{1+\wh{h}_k} \wh{g}_{k-\Delta
    k}.
  \end{equation*}
  Thus, 
  \[
  \norm{\mc{D}\wh{g}}_{2}  \le \norm{
  G_{0}(1+\wh{h})}_{\mc{L}(L^2(\mc{K}))}
  \norm{\frac{\mc{D}\wh{h}}{1+\wh{h}}}_{\infty} \norm{\wh{g}}_{2}
  \le C_{\wh{G}} C_{h,1} C_{V,g,0} :=
  C_{V,g,1}
  \]
  For larger $m$, apply $\mc{D}$ to both sides of
  Eq.~\eqref{eqn:modVeq} for $m$ times, and use
  Lemma~\ref{lem:fVdiff}, we have
  \[
  \sum_{n=0}^{m-1} {m \choose n}(\mc{D}^{(m-n)}\wh{h})_{k}
  (\mc{D}^{(n)}\wh{g})_{k-n\Delta k} + 
  \sum_{l\in \mc{K}}
  \left((\lambda - \wh{h}_{k}) \delta_{kl} -
  \frac{1}{L}\wh{V}_{k-l}\right) (\mc{D}^{(m)}\wh{g})_{l} = 0,\quad k\in \mc{K}.
  \]
  Hence, 
  \begin{equation}
    (\mc{D}^{(m)}\wh{g})_{l} = - \sum_{n=0}^{m-1} {m \choose n}
    \sum_{k\in\mc{K}}\wh{G}_{lk} (1 + \wh{h}_k)
    \frac{(\mc{D}^{(m-n)}\wh{h})_{k}}{1 + \wh{h}_k}
    (\mc{D}^{(n)}\wh{g})_{k-n\Delta k}. 
  \end{equation}
  As $\lVert \wh{G}(1 + \wh{h})\rVert_{\mc{L}(L^2(\mc{K}))}$ is
  bounded, we arrived at the same inequality as in
  \eqref{eq:Dmfbound}. Therefore, the Theorem follows from a same
  induction method as in the proof of Theorem~\ref{thm:Dmboundfinish}.
\end{proof}

\section{Numerical examples}\label{sec:numer}

In this section we demonstrate with numerical experiments the
exponential decay estimate rate for the Green's function associated with
the finite difference (FD) method, and the super-algebraic decay rate
for the Green's function associate with the 
mollified pseudo-spectral (mPS) method.

The mPS scheme is constructed as follows. We mollify the pseudo-spectral
scheme using the mollification function $\wh{\theta}(k)$ in
Eq.~\eqref{eqn:theta} with $\sigma=\frac18$.
%Instead of taking
%$\wh{\theta}(k)$ to be strictly zero when $\abs{k}>\frac34 k_{c}$, we
%take a piecewise constant function $\wh{\theta}_{0}(k)$ that is
%discontinuous at $\pm \frac58 k_{c}$, and convolve $\wh{\theta}_{0}(k)$
%with a bump function, i.e.
%\begin{equation}
%  \wh{\theta} = \wh{\theta}_{0}*\wh{g}, \quad
%%  \wh{g}(k)=\frac{1}{\sqrt{2\pi\sigma^2 k_c^2}} \exp\left(
%%  -\frac{k^2}{2\sigma^2 k_c^2} \right).
%\wh{g}(k)=c_k\exp\left( -\frac{\sigma^2 k_c^2}{\sigma^2 k_c^2-k^2} \right).
%  \label{eqn:theta}
%\end{equation}
%Here $k_{c}=\frac{\pi }{\Delta x}$ and is inversely proportional to the
%real space grid size $\Delta x$, and $c_k$ is a normalization factor so
%that $\int \wh{g}(k)\ud k = 1$. 
One
can verify that the scaling with respect to $k_{c}$ is consistent with
the definition of $\wh{\theta}(k)$ in Eq.~\eqref{eqn:thetacond}.
Fig.~\ref{fig:theta} depicts $\wh{\theta}(k)$ and $\wh{\theta}_0(k)$
for $L=40,\Delta x = 0.02$.

\begin{figure}[ht]
  \begin{center}
    \includegraphics[width=0.3\textwidth]{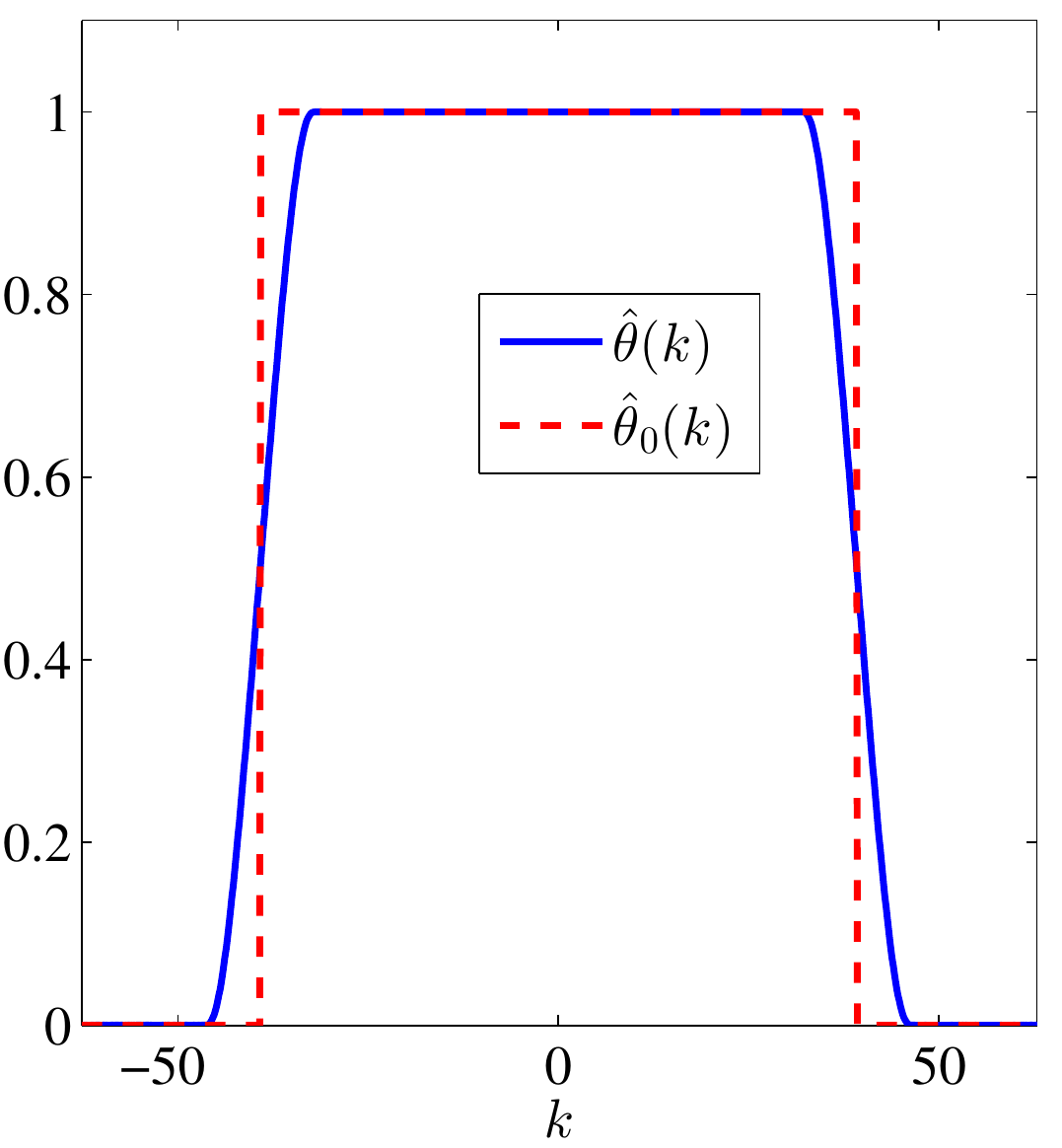}
  \end{center}
  \caption{$\wh{\theta}(k)$ compared with the
  function $\wh{\theta}_{0}(k)$ before smearing.}
  \label{fig:theta}
\end{figure}

First we consider the case when the Hamiltonian contains only the
Laplacian operator, i.e. $\lambda+\Delta$ with $\lambda=-10$. The domain
size $L=40$ and grid size $\Delta x=0.02$. We denote by $G(x,0)$ the
first column of the Green's function ($x\in\mc{X}$).
Fig.~\ref{fig:decaylap} shows $G(x,0)$ for the FD 
discretization decays
exponentially. The discretized Green's function obtained from the pseudo-spectral
method (PS) only decays exponentially up to $10^{-7}$, and then the
decay rate significantly decreases. This transition is
related to the consistency error of the PS scheme, and the transition
can occur at higher accuracy level
by refining $\Delta x$. As discussed in section~\ref{sec:ps}, 
the difficulty for establishing the 
decay properties of the discretized Green's function for the PS method
is that the kernel $k^2$ is not smooth in the Fourier space.
Hence the norms of high order differences of $k^2$ can not be uniformly
bounded.
In contrast, mPS modifies the Laplacian operator so that the diagonal of
the associated kernel is periodic and smooth. Fig.~\ref{fig:decaylap} shows that the Green's
function of mPS indeed decays super-algebraically.

\begin{figure}[ht]
  \begin{center}
    \includegraphics[width=0.3\textwidth]{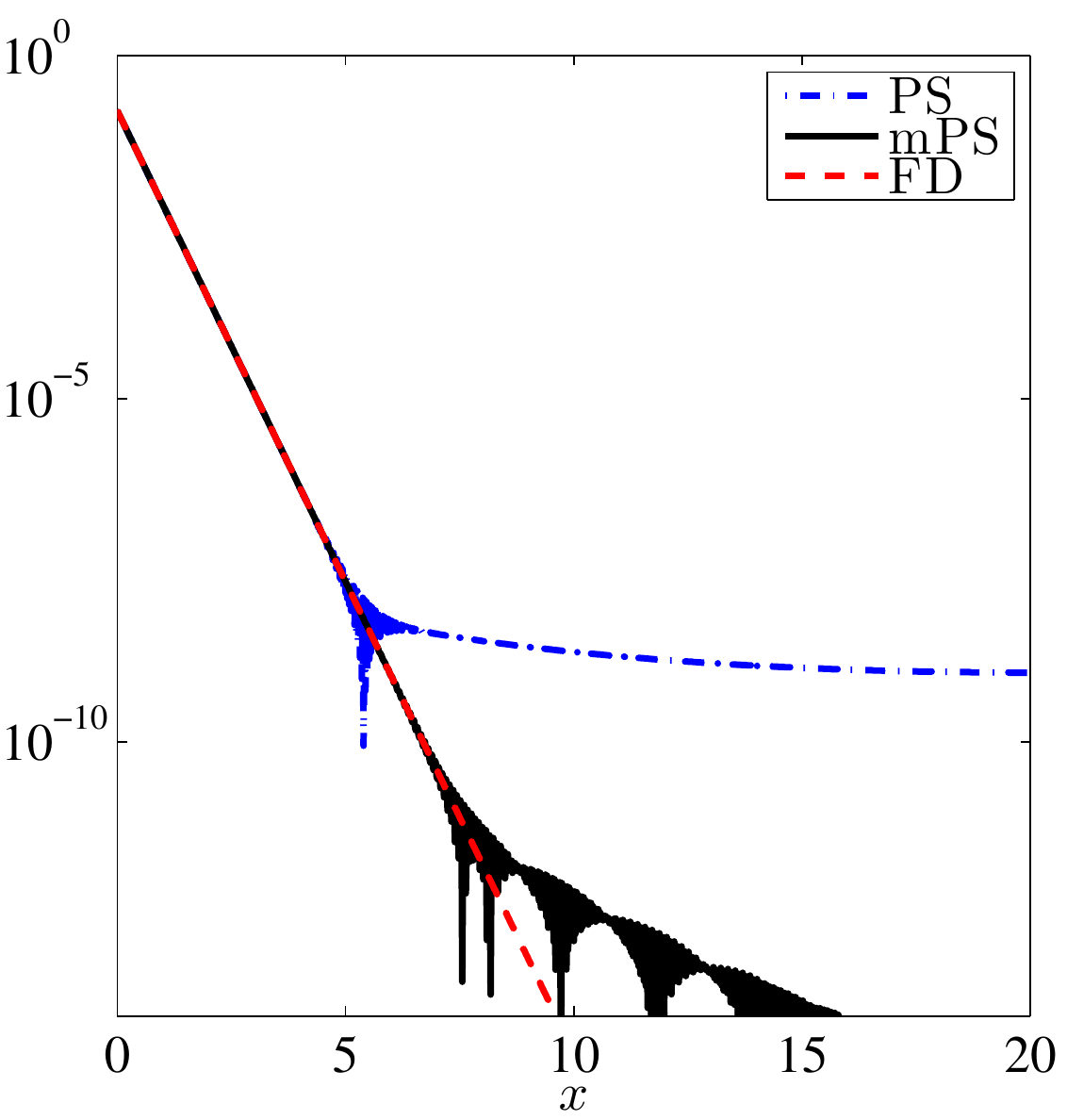}
  \end{center}
  \caption{Decay properties of the one column of the discretized Green's
  function for the operator $\lambda+\Delta$  with $\lambda=-10$.
  Here finite difference (FD), pseudo-spectral (PS) and
  mollified pseudo-spectral (mPS) methods are used. Due to periodic boundary
  condition only $G(x,0)$ for half of the interval $[0,L/2]$ is shown.
  Here $L=40,\Delta x=0.02$.}
  \label{fig:decaylap}
\end{figure}

Next we consider the operator $\lambda+\Delta - V(x)$ where $V(x)$ takes the
form of a Gaussian function which is not band limited, i.e.  $V(x) = 10
e^{-0.2 x^2}$, and $\lambda=-10$. The shape of the potential is shown in
Fig.~\ref{fig:decayVx} (a), and the decay rate for FD, mPS and PS are
given in Fig.~\ref{fig:decayVx} (b). Similar to the Laplacian case, the
addition of the potential function does not modify the behavior of the
decay rate. The off-diagonal elements of Green's function \REV{decay}
exponentially for FD, and super-algebraically for mPS. For PS, the
exponential decay only holds up to the consistency error near $10^{-7}$.

\begin{figure}[ht]
  \begin{center}
    \subfloat[]{\includegraphics[width=0.3\textwidth]{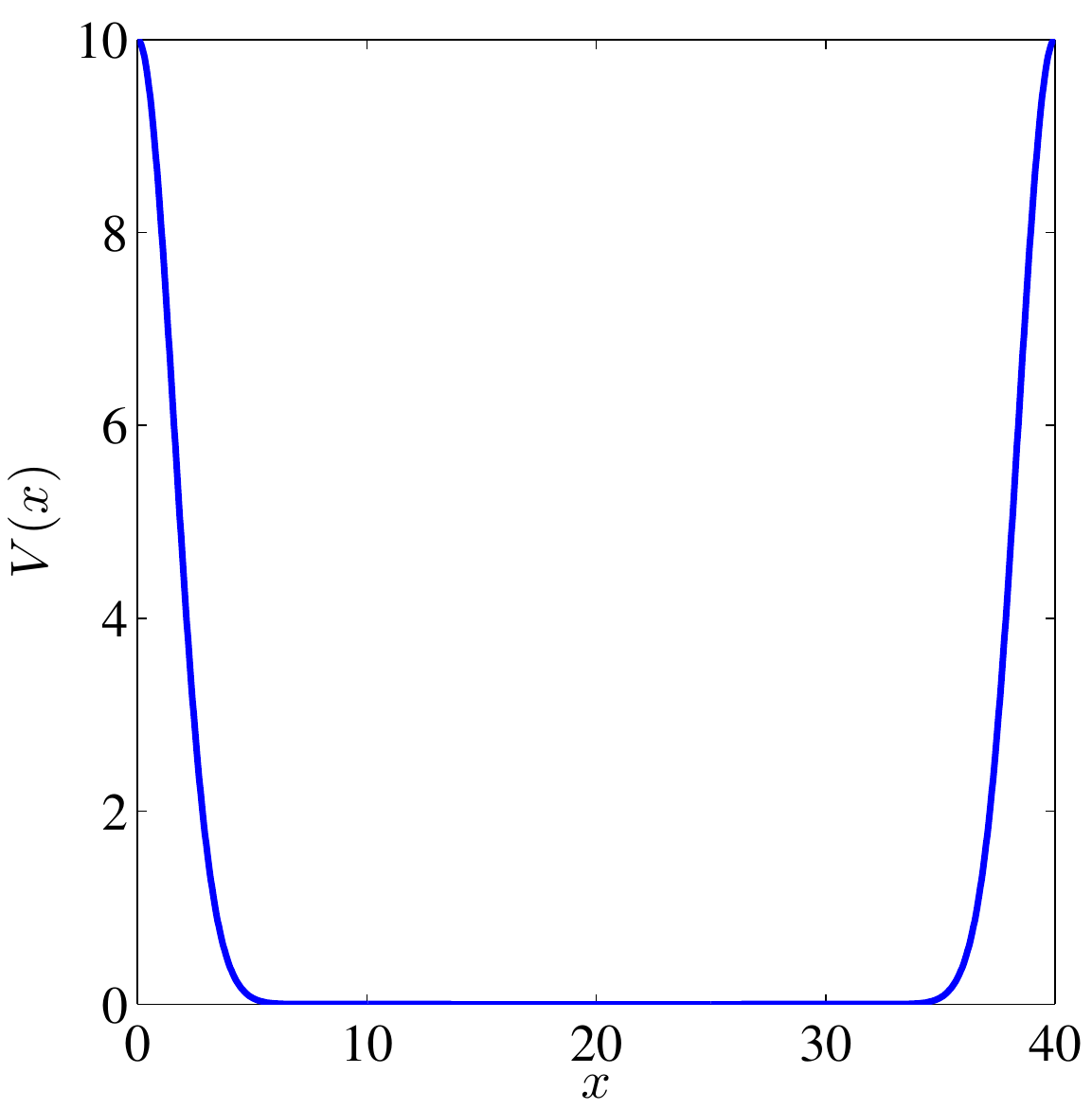}}
    \quad
    \subfloat[]{\includegraphics[width=0.3\textwidth]{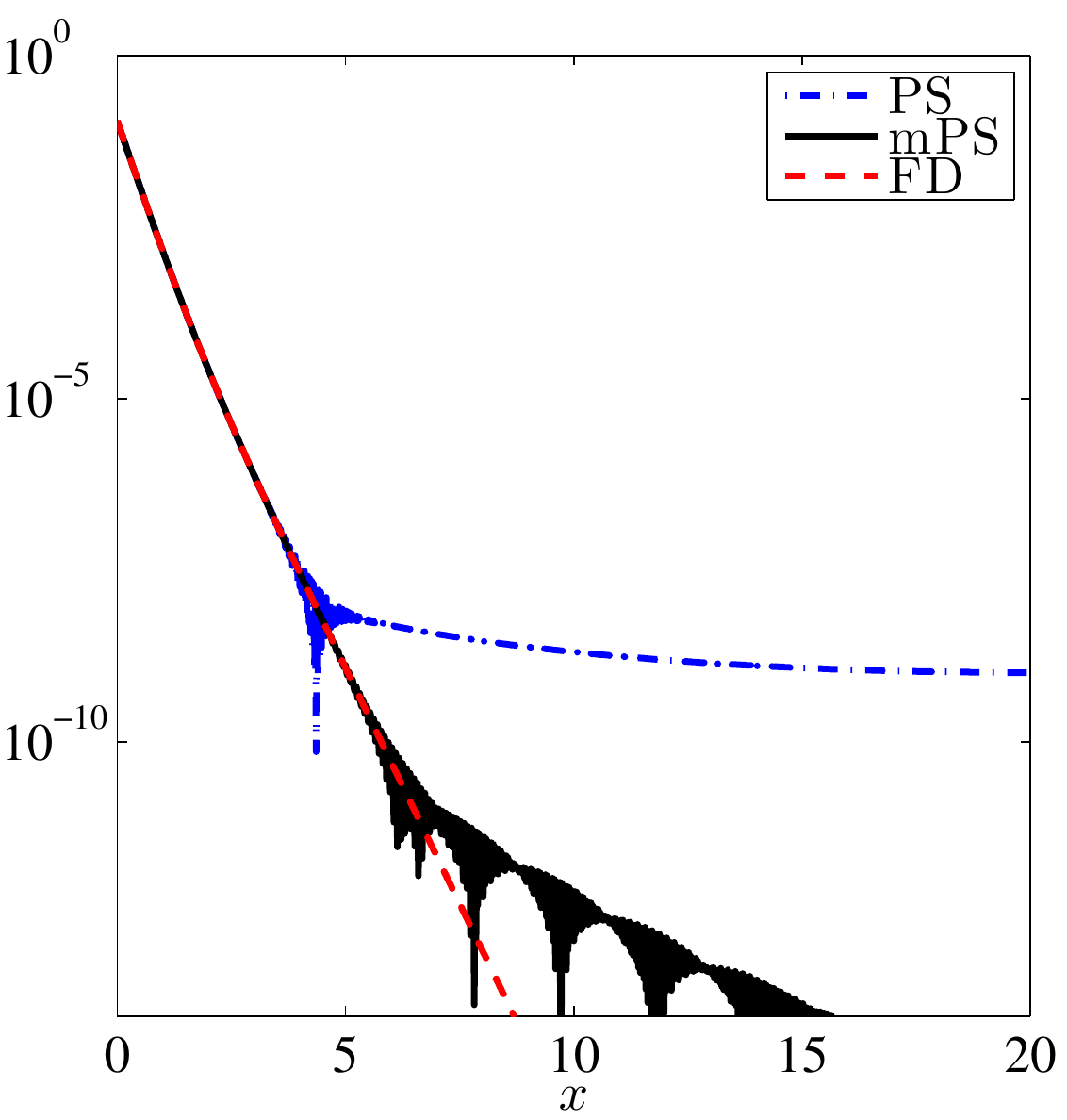}}
  \end{center}
  \caption{(a) Potential function $V(x)$. (b) Decay properties of the
  one column of the discretized Green's function for the operator
  $\lambda+\Delta - V(x)$  with $\lambda=-10$.  Here finite
  difference (FD), pseudo-spectral (PS) and mollified pseudo-spectral
  (mPS) methods are used. Due to periodic boundary condition only
  $G(x,0)$ for half of the interval $[0,L/2]$ is shown. Here $L=40,N=800$.}
  \label{fig:decayVx}
\end{figure}

Below we systematically measure the dependence of the decay rate with
respect to $L$ and $\Delta x$. Although the off-diagonal entries of the
discretized Green's function obtained from the mPS discretization only
\REV{decay} super-algebraically, we expect that the super-algebraic tail is independent of the domain size $L$ with fixed $\Delta x$. We
also expect that the decay behavior will become closer to exponential
decay when $L$ is fixed and $\Delta x$ is decreasing. In order to verify
this, 
consider $\lambda+\Delta-V(x)$ with $\lambda=-10$, and we measure the
exponential decay rate using $G(x,0)$ evaluated at two points $x_1=1.0$
and $x_2=7.0$,  for the mPS  method and FD method, respectively.  
We monitor a quantity $\gamma$ as below
\[
\gamma = -\frac{\log G(x_{2},0)-\log G(x_1,0)}{x_2-x_1}.
\]
$\gamma$ characters the exponential decay rate of $G(x,0)$, and a small
value of $\gamma$ indicates sub-exponential decay.
Fig.~\ref{fig:decayScal} (a) demonstrates the decay rate for
increasingly domain size $L$ from $40$ to $400$, with fixed grid size
$\Delta x = 0.02$. Similarly
Fig.~\ref{fig:decayScal} (b) demonstrate the decay rate for fixed domain
length $L=40$, but with decreasing grid size $\Delta x$ from $0.05$ to
$0.005$. Correspondingly the truncation in the Fourier domain
$k_{c}$ increases from $62.8$ to $628.3$. 
We observe that the decay rate of the finite difference scheme is very
stable and depends very weakly on both $L$ and $k_{c}$. For mPS scheme,
when $\Delta x$ is fixed, the decay rate is lower compared to the decay
rate of the finite difference method. This agrees with the
super-algebraic tail behavior observed in Fig.~\ref{fig:decaylap} and
~\ref{fig:decayVx}. When $L$ is fixed and $\Delta x$ is decreasing and
correspondingly $k_{c}$ is increasing, the decay rate improves as
$k_{c}$ increases, agreeing with our expectation.

%We observe that the
%exponential decay rate depends very weakly on both $L$ and $k_c$.
%Furthermore, the exponential decay rate of mPS and FD agrees very well,
%verifying the sharp decay rate in our estimate.

\begin{figure}[ht]
  \begin{center}
    \subfloat[]{\includegraphics[width=0.32\textwidth]{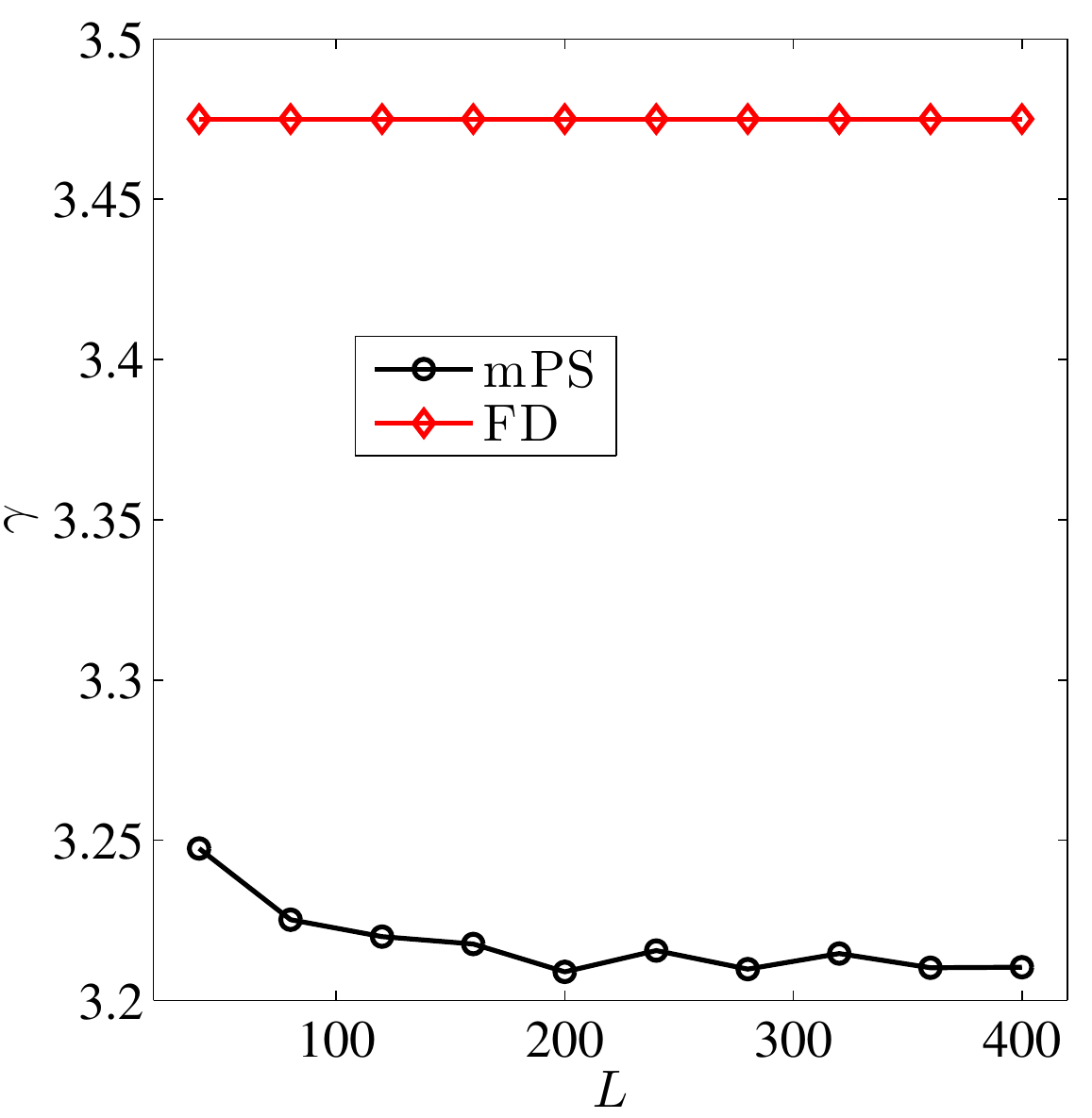}}
    \quad
    \subfloat[]{\includegraphics[width=0.3\textwidth]{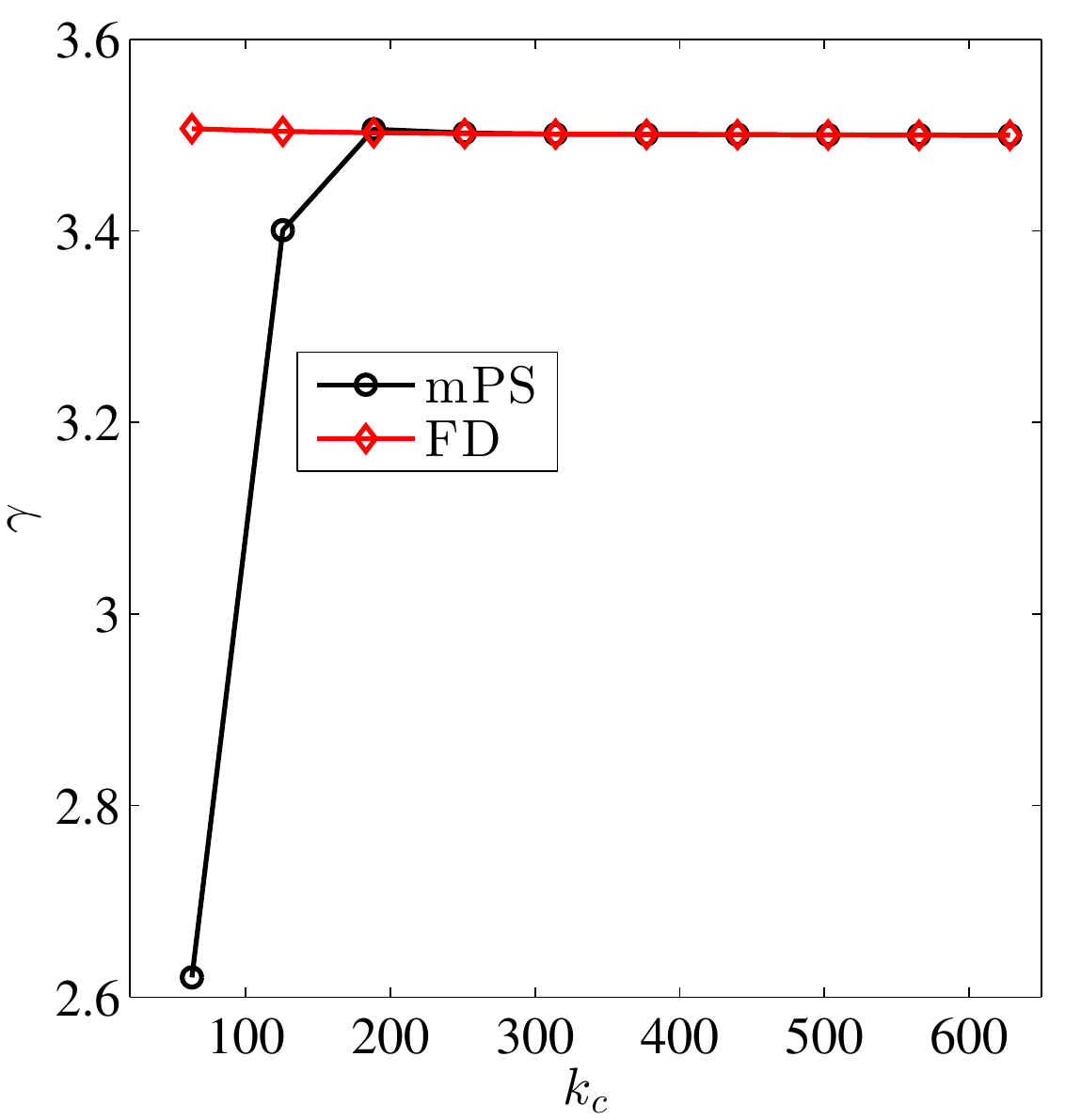}}
  \end{center}
  \caption{For $\lambda+\Delta-V(x)$ with $\lambda=-10$, measure the
  exponential decay rate $\gamma$ for (a) systems with fixed
  $\Delta x = 0.005$ and increasing $L$. (b) systems with fixed
  $L=40$ and increasing $k_{c}$ (and hence decreasing $\Delta x$).} 
  \label{fig:decayScal}
\end{figure}

\section{Conclusion}\label{sec:conclusion}

In this paper, we demonstrate that properly discretized Green's
functions for Schr\"odinger type operators satisfy off-diagonal decay
properties. More specifically, for the finite difference
discretization, the off-diagonal elements of the discretized Green's
function \REV{decay} exponentially. For the mollified pseudo-spectral
discretization, the off-diagonal elements of the discretized Green's
function \REV{decay} super-algebraically. In particular, we obtain
\REV{decay estimates of which the} asymptotic decay rate is
independent of the domain size $L$ and of the discretization parameter
such as the grid spacing.  Our analysis is verified by numerical
experiments for one-dimensional Schr\"odinger type operators. 
\REV{Generalization of our estimate to Schr\"odinger type
operators in higher dimensions is straightforward.}
Our numerical results also indicate that for the widely used
pseudo-spectral discretization, due to the non-smoothness of the
Laplacian operator at the boundary of the Fourier grid, the
asymptotic decay rate of discretized Green's function is only
polynomial with respect to the degrees of freedom. It has been
demonstrated that decay \REV{estimates} of Green's functions can provide
a useful truncation error criterion for designing numerical
schemes~\cite{BenziBoitoRazouk2013}, and our decay estimates can
be useful in correcting such error bound especially for operators with
large spectral radius. \REV{We have assumed uniform grid spacing for both
finite difference and pseudo-spectral discretization. This is most
suited for smooth and bounded potential $V$. The case with unbounded
potential $V$ with isolated singularity points (e.g. in the context of
all-electron calculations) will be studied in the future.}

\section*{Acknowledgments}

The work of L.L.~was partially supported by Laboratory Directed Research
and Development (LDRD) funding from Berkeley Lab, provided by the
Director, Office of Science, of the U.S. Department of Energy under
Contract No.  DE-AC02-05CH11231, the Alfred P. Sloan foundation, the DOE
Scientific Discovery through the Advanced Computing (SciDAC) program and
the DOE Center for Applied Mathematics for Energy Research Applications
(CAMERA) program. The work of J.L.~was supported in part by the Alfred
P. Sloan foundation, the National Science Foundation under awards
DMS-1312659 and DMS-1454939.

% \bib, bibdiv, biblist are defined by the amsrefs package.

%
%\bibliographystyle{siam}
%\bibliography{decay}

\end{document}